\documentclass[oneside,12pt]{article}

\usepackage{amsmath}
\usepackage{amssymb}
\usepackage{pxfonts}
\usepackage{graphicx}
\usepackage{eucal}
\usepackage{mathrsfs}
\usepackage{theorem}
\usepackage{pifont}
\usepackage{color}
\usepackage{enumitem}
\usepackage[backref=page]{hyperref}
\usepackage{minitoc}
\usepackage[normalem]{ulem}
\usepackage[margin=2.5cm]{geometry}

\definecolor{shadecolor}{rgb}{0.8,0.8,0.8}

\usepackage{xcolor} % Required for specifying colors by name
\definecolor{ocre}{RGB}{243,102,25}
\definecolor{darkocre}{RGB}{121,51,12}
\definecolor{lightocre}{RGB}{255,150,37}
\definecolor{verylightocre}{RGB}{255,204,50}
\definecolor{lightgray}{RGB}{200,200,200}
\definecolor{warmblue}{RGB}{51,102,153}
\definecolor{lightwarmblue}{RGB}{105,141,198}
\definecolor{sepia}{RGB}{112,66,20}

\newtheorem{theorem}{Theorem}[section]
\newtheorem{lemma}[theorem]{Lemma}
\newtheorem{proposition}[theorem]{Proposition}
\newtheorem{corollary}[theorem]{Corollary}
\newtheorem{definition}[theorem]{Definition}

{\theorembodyfont{\rmfamily}
}

\theorembodyfont{\rmfamily}
\newenvironment{proof}{{\flushleft \emph{Proof}:}}{\hfill\ding{110}}

\usepackage{framed}
\usepackage{enumitem}
\definecolor{shadecolor}{rgb}{0.90,0.90,0.90}

\newcommand{\e}{\varepsilon}
\newcommand{\R}{\mathbb{R}}
\newcommand{\M}{\mathcal{M}}
\renewcommand{\S}{\mathcal{S}}
\newcommand{\g}{\mathfrak{g}}
\newcommand{\h}{\mathfrak{h}}
\newcommand{\s}{\mathfrak{s}}
\newcommand{\euc}{\mathfrak{e}}
\newcommand{\Vol}{{\operatorname{Vol}}}

\renewcommand{\O}{\operatorname{O}}

\providecommand{\sym}{\operatorname{sym}}

\newcommand{\Def}{\operatorname{Def}}
\newcommand{\dist}{{\operatorname{dist}}}
\newcommand{\SO}{{\operatorname{SO}}}
\newcommand{\Hom}{{\operatorname{Hom}}}
\newcommand{\id}{{\operatorname{Id}}}
\newcommand{\End}{{\operatorname{End}}}
\newcommand{\bbN}{\mathbb{N}}
\newcommand{\bbS}{\mathbb{S}}

\newcommand{\Emph}[1]{{\slshape\bfseries #1}}  % \Emph{.}

\newcommand{\pd}[2]{\frac{\partial#1}{\partial#2}}
\newcommand{\deriv}[2]{\frac{d#1}{d#2}}

\newcommand{\DerivLim}[1]{\left.\frac{D}{\partial #1}\right|_{#1=0}}
\newcommand{\pdLim}[1]{\left.\pd{}{#1}\right|_{#1=0}}
\newcommand{\limn}{\lim_{n\to\infty}}

\newcommand{\W}{\Omega}

\newcommand{\LimEps}{\lim_{\e\to0}}

\newcommand{\LiminfEps}{\liminf_{\e\to0}}

\newcommand{\Rm}{\operatorname{Rm}}
\newcommand{\weakly}{\rightharpoonup}

\newcommand{\Iso}{\operatorname{Iso}}
\newcommand{\iso}{\operatorname{\mathfrak{iso}}}

\newcommand{\Meps}{{\M_\e}}
\newcommand{\Eeps}{E_\e}
\newcommand{\feps}{f_\e}
\newcommand{\hfeps}{\hat{f}_\e}
\newcommand{\vfeps}{\check{f}_\e}
\newcommand{\Feps}{\Psi_\e}
\newcommand{\vFeps}{\check{\Psi}_\e}
\newcommand{\Qeps}{Q_\e}

\newcommand{\Jeps}{A_\e}
\newcommand{\geps}{{\g_\e}}
\newcommand{\VolEps}{\textup{d}\Vol_{\geps}}
\newcommand{\VolS}{\textup{d}\Vol_{\s}}

\newcommand{\hf}{\tilde{f}}
\newcommand{\hPsi}{\tilde{\Psi}}

\newcommand{\Fl}{\operatorname{Fl}}

\newcommand{\ip}[1]{\langle #1 \rangle}

\newcommand{\nabs}{\nabla^\s}

\newcommand{\delnabs}{\delta^{\nabs}}

\newcommand{\n}{\mathfrak{n}}
\newcommand{\Pn}{\mathbb{P}^\n}

\newcommand{\inj}{\hookrightarrow}

\newcommand{\Textand}{\qquad\text{ and }\qquad}

\newcommand{\beq}{\begin{equation}}
\newcommand{\eeq}{\end{equation}}
\newcommand{\bsplit}{\begin{split}}
\newcommand{\esplit}{\end{split}}
\newcommand{\baligned}{\begin{aligned}}
\newcommand{\ealigned}{\end{aligned}}

\newcommand{\brk}[1]{\left(#1\right)}          % \brk{.}     => (.)
\newcommand{\BRK}[1]{\left\{#1\right\}}        % \BRK{.}     => {.}
\newcommand{\Abs}[1]{\left| #1 \right|}        % \Abs{.}     => |.|

\usepackage[all]{xy}
\usepackage{tikz} % Required for drawing custom shapes
\usepackage{tkz-euclide}
\newcommand{\btkz}{\begin{tikzpicture}}
\newcommand{\etkz}{\end{tikzpicture}}

\numberwithin{equation}{section}

\begin{document}

\title{Linearization in incompatible elasticity for general ambient spaces}
\author{Raz Kupferman\footnote{Einstein Institute of Mathematics, The Hebrew University, Jerusalem, Israel}\, and Cy Maor\footnotemark[1]
}

\date{}

\maketitle

\begin{abstract}
Motivated by recent interest in elastic problems in which the target space is non-Euclidean, we study a limit where local rest distances within an elastic body are incompatible, yet close to,  distances within the ambient space.
Specifically, we obtain, via $\Gamma$-convergence, a limit elastic model for a sequence of elastic bodies $(\M,\g_\e)$ in an ambient space $(\S,\s)$, for Riemannian metrics $\g_\e$ and $\s$ such that $\g_\e \to \s$.
Furthermore, we relate the minimum of the limit problem to a linearized curvature discrepancy between $\g_\e$ and $\s$, using recent results of Kupferman and Leder \cite{KL25}.
This relation confirms a linearized version of a long-standing conjecture in elasticity regarding the relation between the elastic energy and the curvature of the underlying space. 
The main technical challenge, compared to other linearization results in elasticity, is obtaining the correct notion of displacement for manifold-valued configurations, using Sobolev truncations and parallel transport.
We show that the associated compactness result is obtained if $(\S,\s)$ satisfies a quantitative rigidity property, analogous to the Friesecke--James--M\"uller rigidity estimate in Euclidean space, and show that this property holds when $(\S,\s)$ is a round sphere.
\end{abstract}

\setcounter{tocdepth}{1}
\tableofcontents

%%%%%%%%%%%%%%%%%%%%%%%%%%%%%%%%%%%%%%%%%%%%%
\section{Introduction}

\paragraph{Incompatible elasticity.} Incompatible elasticity is a theory modeling the mechanics of pre-stressed materials. From a geometrical point of view, pre-stressed materials are characterized by an intrinsic geometry which is incompatible with the geometry of the ambient space, in the sense that they cannot be immersed in the ambient space without strain. Such geometric incompatibilities were first studied in the context of crystalline defects \cite{Nye53,Kon55,BBS55}; the scope of applications has been extended significantly later to include 
complex materials, organic organisms, nano-molecules,  and many more
(see \cite{SMS04,FSDM05,KES07} to name just a few).

In most of the applications, the ambient space is Euclidean two- or three-dimensional space, and it is the elastic body that has an intrinsic non-Euclidean geometry. 
Geometric incompatibility was however also considered in settings where the ambient space is not Euclidean, both in the physics literature, e.g., \cite{KSDM12,AKMMS16} and in the mathematical literature, e.g.,  \cite{GLM14,KMS19,KM25a, CDM24}.
Mathematically, a general model of (hyper-)elasticity can be formulated as follows: the \emph{ambient space} is a complete Riemannian manifold without boundary  $(\S,\s)$. 
The \emph{elastic body} is a compact Riemannian manifold of the same dimension, usually with boundary $(\M,\g)$. 
\emph{Configurations} are maps $f:\M\to \S$. 
To each configuration $f$ corresponds an \emph{elastic energy} $E(f)$, which provides a global quantification of strain, i.e., of how the ``actual" geometry of the embedded body differs from its ``reference" geometry. 
Equilibrium configurations are minimizers of that elastic energy. 
As shown in \cite{KMS19}, this model exhibits an energy-gap in the case $(\M,\g)$ is not isometrically-immersible in $(\S,\s)$ (see below); this gap was quantified in \cite{KM25a} in the asymptotic case in which $\M$ is an asymptotically small ball.
The work \cite{GLM14} also deals with the case of a Riemannian ambient space, but in the compatible case, in which $(\M,\g)$ is a subset of $(\S,\s)$.

A prototypical energy, which we consider in this paper, is 
\beq\label{eq:energy}
E(f) = \int_\M \dist^2(df_p,\SO(\g,f^*\s)|_p)\, \textup{d}\Vol_\g|_p,
\eeq
where 
\[
\SO(\g,f^*\s)|_p  = \BRK{ 
\begin{split} &A: (T_p\M,\g_p)\to (T_{f(p)}\S,\s_{f(p)}) \\
& \text{is an orientation-preserving linear isometry}
\end{split}
},
\]
$\dist$ is the distance in the space of linear maps $T_p\M \to T_{f(p)}\S$ with respect to the metrics $\g_p$ and $\s_{f(p)}$, and $\textup{d}\Vol_\g$ denotes the volume form of $\g$.
That is, the integrand at $p$ denotes the distance squared of $df_p\in \Hom(T_p\M,T_{f(p)}\S)$ from the set of orientation-preserving isometries.
This energy is a Hookean-like energy, i.e., quadratic in the strain and isotropic; we deliberately do not consider more general energy densities, as the main results of this paper are more clearly presented with this one, and the extension to more general energy densities with $2$-growth around the energy well is rather standard.

In coordinates, assuming that $\M$ can be covered by a single chart and $f(\M)$ lies in a single chart in $\S$, this energy density reads as
\[
\dist^2(df,\SO(\g,f^*\s))|_x = \dist^2\brk{\sqrt{\s(f(x))} \circ df(x) \circ \sqrt{\g^{-1}(x)}, \SO(d)},
\]
where we identify $df$, $\g$ and $\s$ with their matrix coordinate representations (of which we take an inverse and a square-root), and $\dist(\cdot,\SO(d))$ on the right-hand side refers to the Frobenius distance between matrices.

As mentioned above, if $(\M,\g)$ cannot be immersed isometrically in $(\S,\s)$, then $\inf E(f) > 0$ \cite{KMS19}; this \emph{energy gap} means that the body is strained even in the absence of body forces or boundary constrains (note that this is a non-trivial result as the energy density is not convex, and thus it is not obvious that minimizers exists). 
Quantification of this energy gap in terms of the underlying geometry is a major open problem in non-Euclidean elasticity, even when $\S$ is Euclidean space.
In that case, some progress was made in \cite{KS12}, and it is conjectured that $\inf E$ is equivalent to a $W^{-2,2}$-norm of the Riemannian curvature tensor of $\g$ \cite{LM22}.

%%%%%%%
\paragraph{Scope of this paper: The ``almost-compatible'' limit.}
The goal of this paper is to investigate this energy gap, for the limit case in which the body manifold can be immersed in the ambient space ``almost-isometrically", and to relate the infimal energy of the limit functional (obtained via $\Gamma$-convergence) to curvature discrepancy between the body and space manifolds.

The setting is as follows: we consider a family of diffeomorphic body manifolds $(\Meps,\geps)$, which are embeddable in $(\S,\s)$; all the manifolds are $d$-dimensional, compact and orientable, and $\S$ is without boundary.\footnote{The analysis in this work carries through also to the Euclidean space $\S=\R^d$, as well as to complete, non-compact spaces that are isometric to the Euclidean space outside a compact subset, that is, that there exists compact sets $K\subset \S$ and $K'\subset \R^d$ such that $\S\subset K$ is isometric to $\R^d\subset K'$, when the latter is endowed with the Euclidean metric.
However, it does not apply to arbitrary non-compact ambient spaces: in this case it might be that \eqref{eq:h} holds, yet $(\M,\geps)$ can be isometrically embedded in $(\S,\s)$ for every $\e$.} 
%some non-compact ambient spaces, like $\S=\R^d$, but this results in some extra assumptions that we would rather avoid.} 
Without loss of generality, the manifolds $\Meps$ can be identified with an open subset $\M\subset\S$. 
The body manifolds $(\M,\geps)$ are almost-compatible with the ambient space in the sense that $\geps\to\s$ uniformly as $\e\to0$;  moreover, the parameter $\e$ quantifies the decreasing scale of metric discrepancy, in the sense that 
\beq\label{eq:h}
\lim_{\e \to 0} \frac{\geps-\s}{\e} = \h \qquad \text{strongly in $L^2$,}
\eeq
where $\h$ is a symmetric section of $T^*\M\otimes T^*\M$.

In this regime, the elastic energy is expected to scale as $\e^2$. Thus, we introduce rescaled energy functionals
\beq\label{eq:Eeps}
\Eeps(\feps) = \frac{1}{\e^2} \int_\M \dist^2(d\feps,\SO(\geps,\feps^*\s))\, \VolEps.
\eeq
Our aim is to obtain the $\Gamma$-limit of these functionals, as well as an associated compactness result, and to analyze the minimal energy.
Our main results can be summarized as follows: 

\begin{theorem}
\label{thm:main}
Let $(\S,\s)$ be a closed manifold,\footnote{Or spaces as described in footnote 1.} and let $\M\subset \S$ be an open, bounded subset with Lipschitz boundary, endowed with metrics $\g_\e$ satisfying \eqref{eq:h}.
Then the following hold:
\begin{enumerate}
\item \textbf{$\Gamma$-convergence:} The energies $\Eeps$ given by \eqref{eq:Eeps} $\Gamma$-converge as $\e\to0$ to a limit energy
\[
E_0 : W^{1,2}\mathfrak{X}(\M) \to [0,\infty),
\]
given by
\beq
\label{eq:E_0}
E_0(u) = \frac{1}{4} \int_\M |\mathcal{L}_u\s - \h|^2\,\VolS,
\eeq
where $W^{1,2}\mathfrak{X}(\M)$ denotes the space of vector fields on $\M$ having $W^{1,2}$-regularity, and $\mathcal{L}_u\s$ is the Lie derivative of the metric $\s$ along the vector field $u$.
The convergence of the configurations $f_\e$ to a limit displacement $u$ is given in Definition~\ref{def:limit} and is also discussed below.
\item \textbf{Compactness:} If $(\S,\s)$ satisfies a quantitative rigidity property (Definition~\ref{def:FJM} below), then $E_\e(f_\e)\lesssim 1$ implies that $f_\e\to u$ in the sense of Definition~\ref{def:limit}.
This holds, in particular, for the case where $(\S,\s)$ is the Euclidean space or the round sphere, and for a closed manifold $(\S,\s)$ when $\M=\S$.
\item 
\textbf{Minimal energy and curvature:} Let $P_\perp \h$ be the projection of $\h$ on the space orthogonal to Lie derivatives. 
Then we have the following lower bound,
\[
\min E_0 \gtrsim \|\dot{\Re}_\s P_\perp \h\|^2_{W^{-2,2}(\M)},
\]
where $\dot{\Re}_\s \sigma$ is the curvature variation when perturbing the metric $\s$ in direction $\sigma$.
This bound is tight if $\M$ is simply-connected, and $\s$ is a metric of constant non-negative curvature or a small compactly-supported $C^2$-perturbation thereof.
\end{enumerate}
\end{theorem}

Part 1 of the theorem is formulated precisely in Theorem~\ref{thm:Gamma_conv}, Part 2 in Proposition~\ref{prop:3.1} and Corollary~\ref{corr:sphere}, and Part 3 in Theorem~\ref{thm:energy_estimate}.
In the energy $E_0$, the Lie derivative of the metric is given by
\beq
\label{eq:calL}
\mathcal{L}_u\s(X,Y) = \s(\nabla_X^\s u, Y) + \s(X,\nabla_Y^\s u),
\eeq
and the norm $|\cdot|$ refers to the norm of $(2,0)$-tensors induced by $\s$.
$\mathcal{L}_u\s$ is the Riemannian generalization of the classical linearized strain (symmetric gradient of the displacement) in linear elasticity.

This limit is essentially a linearization of the energy \eqref{eq:Eeps} for the case of weak incompatibility.
Linearization in elasticity via $\Gamma$-convergence has been studied extensively in the last two decades (see, e.g., \cite{DNP02,PT09,MPT19,JS21,MM21,KM25bb}), however this work is, to the best of our knowledge, the first dealing with a non-Euclidean target space.
The fact that the target is non-Euclidean results in some new challenges in defining the notion 
of a limit displacement field, as well as in the rigidity estimate needed for the compactness result, as discussed below.
We expect the tools developed here, in particular the displacement notion, to be useful in other low-energy manifold-valued problems.

%%%%%%%
\paragraph{The notion of displacement.}
When the ambient space is Euclidean, one defines the rescaled displacements associated with the configurations $\feps$ via 
\beq\label{eq:ueps_Euc}
u_\e = \frac{\Feps^{-1}\feps - \id}{\e},
\eeq
where $\Feps$ is the isometry of the ambient Euclidean space closest (usually in $W^{1,2}$)  to $\feps$.
The topology of the $\Gamma$-convergence is then defined via a weak $W^{1,2}$ convergence of these displacements, where the limit displacement $u$ should be thought of as a vector field along the elastic body.
The limit energy is a linear-elastic energy of the limit displacement $u$. 
A crucial step in proving the associated compactness result, namely the existence of $\Feps$ close enough to $\feps$, is obtained via the celebrated Friesecke--James--M\"uller (FJM) geometric rigidity estimate \cite{FJM02b}, as discussed below.

In our case the ambient space is generally non-Euclidean; as such, there are challenges in defining a notion of convergence, as  \eqref{eq:ueps_Euc} is not defined for manifold-valued maps.
We define a notion of convergence of maps $\feps \in W^{1,2}(\M,\S)$ to a limit displacement field $u\in W^{1,2}\mathfrak{X}(\M)$ as follows. 
Denote by $\Iso(\S)$ the group of isometries of $\S$ (which may be trivial). 
Roughly speaking, we say that $\feps$ converges to a pair $(\Psi,\nabs u)$, where $\Psi\in\Iso(\S)$, $u\in  W^{1,2}\mathfrak{X}(\M)$, and $\nabs$ is the Riemannian connection of $\S$, if there exists a family of isometries $\Feps\in\Iso(\S)$ converging to $\Psi$, such that the following hold:
\begin{enumerate}[itemsep=0pt,label=(\alph*)]
\item The $L^2$-distance between $\feps$ and $\Feps$ is of order $\e$.
\item The distance between $d\feps$ and $d\Feps$ is to leading order (in a weak $L^2$ sense) equal to $\e\,d\Feps(\nabs u)$.
\end{enumerate}
This characterization is somewhat imprecise, as $d\feps|_p: T_p\M \to  T_{\feps(p)}\S$ and $d\Feps|_p:T_p\M \to T_{\Feps(p)}\S$ are tensors over different fibers. 
The precise definition (Definition~\ref{def:limit}) requires one to parallel-transport $d\feps|_p$ from the fiber $T_p^*\M\otimes T_{\feps(p)}\S$ to the fiber $T_p^*\M\otimes T_{\Feps(p)}\S$, which is the reason why the limit map is a covariant derivative and not a differential.
Introducing parallel transport presents a complication as it is not guaranteed that $\Feps(p)$ is within the injectivity radius of $\feps(p)$. 
To resolve this difficulty, the definition of convergence $\feps\to(\Psi,\nabs u)$ has to include as a first step a Lipschitz truncation, replacing $\feps$ with  a family of uniformly-Lipchitz configurations $\hfeps\in W^{1,\infty}(\M;\S)$, which approximate $\feps$ as defined in Section~\ref{subsec:truncation}.
The  idea of using Sobolev truncations to define a notion of convergence for manifold-valued problems was first used by Kr\"omer and M\"uller \cite{KM25a}.
One then needs to prove that the limit of the transported maps is indeed a  $\nabla^\s$-covariant derivative of a vector field $u$ (Proposition~\ref{prop:J_is_a_gradient}).
The vector field $u$, which corresponds to a rescaled infinitesimal displacement, is unique if the isometry group of $\S$ is discrete; otherwise, it is unique modulo an element of the Lie algebra $\iso(\S)$, which is immaterial in the limit energy (Proposition~\ref{prop:limit_unique}).

%%%%%%%%%%%%
\paragraph{Compactness and geometric rigidity.}

Having established a notion of convergence, we obtain a $\Gamma$-convergence result as described in Theorem~\ref{thm:main}(1).
To obtain that (approximate) minimizers of $\Eeps$ converge (modulo a subsequence) to a minimizer of $E_0$, we need to prove a compactness property, namely, that every sequence $\feps$ of bounded energy has a converging subsequence. In the Euclidean setting (i.e., $(\S,\s)$ being the Euclidean space), every low-energy compactness result hinges on the FJM rigidity estimate \cite{FJM02b}, which asserts that for every bounded domain with Lipschitz boundary, to every $f\in W^{1,2}(\M;\R^d)$ corresponds an isometry $\Psi$ of $\R^d$, such that
\[
\|f - \Psi\|_{W^{1,2}} \le C \,\|\dist(df,\SO(d))\|_{L^2},
\]
where the constant $C$ depends only on the domain.
Generalizing the FJM rigidity theorem for Riemannian manifolds is a widely-open problem. 
To obtain the sought compactness result (Proposition~\ref{prop:3.1}), we need the space manifold $(\S,\s)$ to satisfy the following property:

%%%%%%%%%%%%%
\begin{definition}
\label{def:FJM}
A space $(\S,\s)$ is said to satisfy a \Emph{quantitative rigidity property}, if for every Lipschitz subdomain $\M\subset \S$ there exists for every $f \in W^{1,2}(\M;\S)$ an isometry $\Psi$ of $(\S,\s)$, such that 
\[
\|\kappa\circ f - \kappa\circ \Psi\|_{W^{1,2}(\M;\R^{D})} \le C\, \|\dist(df,\SO(\s,f^*\s))\|_{L^2(\M)},
\]  
where $\kappa:\S\to \R^D$ is some isometric embedding, and the constant $C$ in the inequality may depend on $\M,\S$ and $\kappa$ but not on $f$.
\end{definition}
%%%%%%%%%%%%%

We prove that round spheres satisfy this quantitative rigidity property (Theorem~\ref{thm:AKM25b}); this result strengthens \cite[Theorem~3.2]{CLS22}, in which the constant depends on higher norms of the map $f$.
Thus the limit weak-incompatibility model \eqref{eq:E_0} applies at least to spheres (and to Euclidean space).
We conjecture that any symmetric space satisfies this quantitative rigidity property, and that it holds for any ``generic'' (in some appropriate sense) metric on a given closed manifold.

Recently, after the completion of this manuscript, Conti, Dolzmann and M\"uller \cite{CDM24} proved that every closed manifold $(\S,\s)$ satisfies the above quantitative rigidity property for the case $\M = \S$, and hence our compactness result holds in this case as well (See Theorem~\ref{thm:CDM} below).
However, this quantitative rigidity property definitely does not hold in this generality if one allows $\M$ to be a strict subset of $\S$: one could easily construct a manifold $(\S,\s)$ with a trivial isometry group, yet in which an open set $\M\subset \S$ can be isometrically immersed in non-trivial ways (see Section~\ref{sec:discussion} for details).

In order to (potentially) include such manifolds, one can relax Definition~\ref{def:FJM} by requiring that $\Psi$ is an isometric immersion $\Psi:\M\to \S$ rather than $\Psi \in \Iso(\S)$; in this case we conjecture that this property holds for any closed manifold.
This relaxation of the quantitative rigidity property, however, does not guarantee compactness given how we defined the convergence of displacements (Definition~\ref{def:limit}).
One can modify this notion of convergence as well, but this complicates the analysis considerably. 
In order not to obscure the essential parts of the analysis, and since the only spaces $(\S,\s)$ we currently know to satisfy a quantitative rigidity property satisfy it in the strong form, we opted not to take this route.
In Section~\ref{sec:discussion} we elaborate on how the notion of convergence can be relaxed and how this affects the rest of the analysis. 

%%%%%%%%%
\paragraph{The linear elastic limit as a measure of linearized curvature discrepancy.}
In Section~\ref{sec:energy} we analyze the limit energy $E_0$. As exposed above, the obstruction for a zero-energy state is metric incompatibility, and thus the minimum of $E_0$ must depend on the leading-order term $\h$ of the metric discrepancy. Note that non-zero $\h$ does not preclude the existence of zero-energy states, as an isometric flow may lead to a non-zero $\h$. 

Much of the result in this section hinge on recents results \cite{KL22,KL25}.
The first observation is that $\h$ can be decomposed orthogonally into 
\[
\h = P_\parallel\h + P_\perp\h, 
\]
where $P_\parallel$ is an orthogonal projection into the image of the map $u\mapsto \mathcal{L}_u\s$ (known as the deformation map, or as the Killing operator), and $P_\perp$ is the projection into its orthogonal complement. The minimizer of $E_0$ is then obtained for a displacement $u$ that cancels the component $P_\parallel\h$ of the metric discrepancy. 
As we explain in Section~\ref{sec:energy}, the fact that the minimum is obtained for a non-zero displacement reflects the fact that the identification of body manifolds with a fixed submanifold $\M\subset\S$, which can be viewed as a choice of a reference configuration, is not necessarily optimal energetically, but only provides the right energetic scale; thus the notion of displacement involves an arbitrary choice. 
By modifying these reference configurations, we could have required from the outset that $\h$ be orthogonal to the range of the deformation operator, which would have implied that $u=0$ is the minimizer of the limit energy.

We further show that 
\[
\min E_0 = \tfrac14  \|P_\perp\h\|^2_{L^2(\M)} \simeq  \|\dot{\Re}_\s P_\perp \h\|^2_{W^{-2,2}(\M)} + \|\mathscr{P} \h\|^2_{L^2(\M)},
\]
where $\dot{\Re}_\s \sigma$ is the curvature variation when perturbing the metric $\s$ in direction $\sigma$, and $\mathscr{P} \h$ is the projection of $\h$ onto a finite dimensional module of smooth functions (the finite-dimensional obstruction associated with the relevant overdetermined boundary-value problem).
This module is trivial if $\M$ is simply-connected and $\s$ is of constant non-negative sectional curvature, or a small perturbation thereof.
In general, it is not known whether this is true for any metric on a simply-connected domain.

To conclude, at least in the case of a simply-connected body manifold, and an ambient space with a metric that is sufficiently close to a constant, non-negative curvature metric, the minimal energy is of the order of the square of the $W^{-2,2}$-norm of an asymptotic deviation of the Riemann curvatures of the domain metrics (after choosing an appropriate gauge/reference configuration) from the Riemann curvature of the ambient space (Theorem~\ref{thm:energy_estimate}).
The relation of the minimal energy with a $W^{-2,2}$-norm of a curvature discrepancy has been conjectured by Lewicka and Mahadevan in the case when the space manifold is Euclidean \cite{LM22}; our result thus affirms a linearized version of this conjecture.

This result is also related to the analysis of Kr\"omer and M\"uller of the elastic energy of asymptotically small balls into an ambient space $(\S,\s)$ \cite{KM25a}.
In this work, the sequence of elastic bodies are balls $B_\e(p)\subset \M$, $\e\to0$, where $(\M,\g)$ is a Riemannian manifold.
After rescaling, they obtain that the infimal limit energy is a curvature discrepancy between the curvature of $(\M,\g)$ at $p$, and the curvature of $(\S,\s)$ at the point into which $p$ is mapped.
Since small balls in a Riemannian manifold are ``almost Euclidean'', and those small balls map (after appropriate Sobolev truncations) to small domains in $(\S,\s)$, we can view the result of \cite{KM25a} as the case of a fixed domain endowed with a sequence of metrics converging to a Euclidean metric, mapped into a manifold whose metric is also converging to a Euclidean metric, albeit not in the same way (hence the curvature discrepancy). 
The fact that both metrics are asymptotically Euclidean is a key difference with the current paper: as a result of this difference, in \cite{KM25a} the definition of the limit displacement does not require parallel transport, the Euclidean FJM rigidity theorem suffices for compactness, and the analysis of the limit energy is simpler. 
Their notion of convergence still requires Sobolev truncation, similar to this paper.

\paragraph{Notation}
Throughout this work, for a manifold $\M$, we denote by $\mathfrak{X}(\M)$ the module of vector fields on $\M$, and for a vector bundle $E\to\M$, we denote by $\Gamma(E)$ sections of $E$ and by $\W^1(\M;E)$ the module of $E$-valued $1$-forms. We use notations such as $L^2\Gamma(E)$, $W^{k,p}\W^1(\M;E)$, etc.\ for Sobolev versions of those spaces. Let $f:\M\to\S$ be a map between manifolds and let $G\to\S$ be a vector bundle; we denote by $f^*G$ the pullback bundle over $\M$, with the canonical identification $f^*G_p \simeq G_{f(p)}$. We use the same standard notation for the pullback of sections. 

Let $(\S,\s)$ be a Riemannian manifold. We denote by $d_\s:\S\times\S\to\R$ the distance function induced by the Riemannian metric $\s$. 
We denote by $\nabs$ the Riemannian connection on $T\S$, and retain the same notation for the induced connections on all tensor products of $T\S$ and $T^*\S$. We further denote by $\Iso(\S)$ the group of isometries of $\S$, which is a finite-dimensional Lie group (possibly trivial); we denote by $\iso(\S)$ the associated Lie algebra of Killing fields. These are the vector fields $u$ that satisfy $\mathcal{L}_u \s = 0$, and are the Riemannian equivalent of vector fields whose symmetric gradient vanishes.

Let $V$ be a vector field and let $A\in V^*\otimes V^*$ be a bilinear form. Then, naturally, its symmetrization $\sym A$ is defined by
\[
\sym A(u,v) = \tfrac12(A(u,v) + A(v,u)).
\]
In contrast, if $V$ is an inner-product space, and $B\in V^*\otimes V \simeq \Hom(V,V)$ is an endomorphism of $V$, then its symmetrization $\sym B\in\Hom(V,V)$ with respect to the inner-product is defined via
\[
\ip{\sym B(u),v} = \tfrac12(\ip{Bu,v} + \ip{u,Bv}).
\]
The musical isomorphism $\flat: V^*\otimes V \to V^*\otimes V^*$ is the identification of an endomorphism with a bilinear form via the inner-product, i.e.,
\[
B^\flat(u,v) = \ip{Bu,v}.
\]
It is obvious that $\sym$ and $\flat$ commute, i.e.,
\[
\sym B^\flat = (\sym B)^\flat.
\]
These definitions generalize to tensors on Riemannian manifolds, i.e., to $(1,1)$-tensors $\xi\in \Gamma(T^*\M\otimes T\M)$ and $(2,0)$-tensors $h\in \Gamma(T^*\M\otimes T^*\M)$. 
Specifically, for a Riemannian manifold $(\M,\s)$ and vector field $u\in \mathfrak{X}(\M)$, the Lie derivative \eqref{eq:calL} of $\s$ in direction $u$ is 
\[
\mathcal{L}_u\s = 2(\sym \nabs u)^\flat.
\]

Finally, we use the symbols $\lesssim$ and $\gtrsim$ to denote inequalities up to a multiplicative constant, i.e.,
\[
f(x) \lesssim g(x)
\]
means that there exists a constant $C>0$ such that $f(x) \le C\, g(x)$ for all $x$. 
If $f(x)\lesssim g(x)$ and $f(x)\gtrsim g(x)$, we write $f(x)\simeq g(x)$.  

\paragraph{Acknowledgements}
We are grateful to Or Hershkovits, Roee Leder and Jake Solomon  for fruitful discussions. 
We are also grateful to Stefan M\"uller for communicating to us the recent result in \cite{CDM24}.
RK was partially funded by ISF Grant 560/22 and CM was partially funded by ISF Grants 1269/19 and 2304/24 and BSF grant 2022076.

%\paragraph{Data availability statement}
%Not applicable.

%%%%%%%%%%%%%%%%%%%%%%%%%%%%%%%%%%%%%%%%%%%%%
\section{Convergence of displacements}

The energy \eqref{eq:Eeps} is defined over the Sobolev space $W^{1,2}(\M;\S)$.
This space is defined by 
\[
W^{1,2}(\M;\S) ~:~ \BRK{f:\M\to \S ~:~ \kappa\circ f \in W^{1,2}(\M;\R^D)}, 
\]
where $\kappa:(\S,\s)\to \R^D$ is a fixed isometric embedding for some $D$ large enough;  $W^{1,2}(\M;\S)$ inherits the strong and weak topologies of $W^{1,2}(\M;\R^D)$.
The definition of the space and the topologies induced are independent of the choice of $\kappa$ (since $\S$ is compact), and the intrinsic and extrinsic distances on $\kappa(\S)$ are equivalent.
Finally, we can associate with $f\in W^{1,2}(\M;\S)$ a weak derivative $df$, which is defined for almost every $p\in \M$ as a map $T_p\M\to T_{f(p)}\S$, such that $d(\kappa\circ f) = d\kappa\circ df$ \cite{CS16}.
Strictly speaking, since for $p\in\M$,  $d(\kappa\circ f)_p = d\kappa_{f(p)}\circ df_p$, we should write $d(\kappa\circ f) = f^*d\kappa\circ df$; we opt for the less formal notation for the sake of readability, but one should always keep in mind the point at which $d\kappa$ is being evaluated.

%%%%%%%%%%%%%%%%%%%%%%
\subsection{Sobolev truncation}
\label{subsec:truncation}

There are several difficulties when working with $W^{1,2}(\M;\S)$ maps: 
They may fail to be localizable, i.e., the images of arbitrarily small neighborhoods of a point $p\in \M$ may fail to lie in a coordinate chart around $f(p)\in \S$.
Moreover, for two maps $f,f':\M\to\S$, we cannot directly compare $df(p):T_p\M \to T_{f(p)}\S$ and $df'(p):T_p\M \to T_{f'(p)}\S$, as these maps belong to different fibers of a vector bundle.
To remedy this, one would like to parallel transport $T_{f(p)}\S$ to $T_{f'(p)}\S$ along a minimal geodesic; however, even if the maps are smooth and localizable, a minimal geodesic is not necessarily unique if $f(p)$ is far enough from $f'(p)$.
As we will see below, in the context of sequences of low energy configurations $\feps$, the remedy is to approximate $\feps$ by a sequence $\hfeps$ of uniformly Lipschitz functions, a procedure known in the literature as Lipschitz approximation or truncation (cf. \cite{AKM25b}):

%%%%%%%%%%%%%
\begin{definition}
\label{def:truncation}
Let $\feps\in W^{1,2}(\M;\S)$. A sequence $\hfeps\in W^{1,\infty}(\M;\S)$ is said to be 
an \emph{$\e$-truncation} of  $\feps$ if the following hold:
\begin{enumerate}[itemsep=0pt,label=(\alph*)]
\item $\kappa\circ\hfeps$ is uniformly-bounded in $W^{1,\infty}(\M;\R^D)$. 
\item $\Vol_\s(\{\hfeps\ne \feps\}) \lesssim \e^2$.
\item $\|\kappa\circ\hfeps - \kappa\circ\feps\|_{W^{1,2}} \lesssim \e$. 
\end{enumerate}
\end{definition}
%%%%%%%%%%%%%

While in general, Lipschitz maps are not dense in $W^{1,2}(\M;\S)$, the next proposition shows that an $\e$-truncation always exists for a sequence of maps $\feps$ having bounded energy; moreover, it can be chosen with the constant in the bound (a) independent of the sequence $\feps$.
A similar analysis first appeared, for manifold-valued maps, in \cite{KMS19}; see also \cite{KM25a,AKM25b} for similar constructions. All are based on the Euclidean result \cite[Proposition~A.1]{FJM02b}.

%%%%%%%%%%%%%
\begin{proposition}
%\label{prop:}
Let $\feps \in W^{1,2}(\M;\S)$ satisfy $\Eeps(\feps) \lesssim 1$, where $\Eeps$ is defined in \eqref{eq:Eeps}. Then there exists an $\e$-truncation $\hfeps\in W^{1,\infty}(\M;\S)$ of $\feps$.
Moreover, the constant in Property (a) can be chosen independently of the sequence $\feps$, depending only on $\M$, $\S$ and $\kappa$.
\end{proposition}
%%%%%%%%%%%%%

This proposition follows immediately from Lemmas~\ref{lem:KMS19}--\ref{lem:almost_the_same}; 
the first, and main one, follows from the analysis in \cite[Sec.~3.2, Step IV]{KMS19}:

%%%%%%%%
\begin{lemma}
\label{lem:KMS19}
Let $\feps \in W^{1,2}(\M;\S)$. Then, there exists a family $\hfeps\in W^{1,\infty}(\M;\S)$ satisfying  the following:
\begin{enumerate}[itemsep=0pt,label=(\alph*)]
\item $\|\kappa\circ\hfeps\|_{W^{1,\infty}}\le C$ for some $C>0$ depending only on $\M$, $\S$ and $\kappa$ (and not on $f_\e$). 
\item $\Vol_\s(\{\hfeps\ne \feps\}) \lesssim \int_\M \dist^2(d\feps,\SO(\s,\feps^*\s))\,\VolS$.
\item $\|\kappa\circ\hfeps - \kappa\circ\feps\|^2_{W^{1,2}} \lesssim \int_\M \dist^2(d\feps,\SO(\s,\feps^*\s))\,\VolS$. 
\end{enumerate}
\end{lemma}
%%%%%%%

The right-hand sides in Items (b) and (c) do not coincide with $\e^2\,\Eeps(\feps)$ since the bundle of isometries is $\SO(\s,\feps^*\s)$ rather than $\SO(\geps,\feps^*\s)$. 
However, as shown below in Lemma~\ref{lem:almost_the_same}, the difference is
$O(\e^2)$ as the metrics $\geps$ and $\s$ are asymptotically $\e$-close.
The next lemma will be used both in Lemma~\ref{lem:almost_the_same} and later in the $\Gamma$-convergence analysis.

%%%%%%%%%%%%%
\begin{lemma}
\label{lem:Qeps}
Denote by $\SO(\g_\e,\s)\subset \End(T\M)$ the bundle of linear isometries $(T\M,\g_\e) \to (T\M,\s)$.
There exists an $L^2$-section
$\Qeps$ of $\SO(\geps,\s)$, such that
\beq
\Qeps \to \id_{T\M} \qquad \text{uniformly,}
\label{eq:Qeps0}
\eeq
\beq
\frac{\Qeps - \id_{T\M}}{\e} \to \xi,
\qquad
\text{in $L^2\W^1(\M;T\M)$},
\label{eq:Qeps}
\eeq
where $\xi\in L^2\W^1(\M;T\M)$ satisfies $\h = 2\,\sym \xi^\flat$, namely for $u,v\in T\M$, 
\beq\label{eq:xi_h}
\h(u,v) = (\xi(u),v)_\s + (u,\xi(v))_\s.
\eeq
\end{lemma}
%%%%%%%%%%%%%

Note that even if $\h$ is smooth, $\Qeps$ may  fail to even be continuous, as $\M$ may topologically preclude continuous non-vanishing vector fields. 

%%%%%%%%%%%%%
\begin{proof}
Since the regularity of $\Qeps$ and $\xi$ is $L^2$, we may construct them by partitioning $\M$ into disjoint  coordinate patches.
Let $U$ be a coordinate neighborhood of $p$. Identifying $\s$ and $\geps$ with their coordinate matrix representations, let
\[
\Qeps = \s^{-1/2} \geps^{1/2}, 
\]
which is a section of $\SO(\geps,\s)$, and the square roots are the symmetric positive-definite square roots of a matrix (which are coordinate-dependent operations). The fact that $\geps\to\s$ uniformly implies that $\Qeps\to\id_{T\M}$ uniformly.

Now,
\[
\begin{split}
\frac{\Qeps - \id_{T\M}}{\e} = \frac{(\id_{T\M} + \s^{-1} (\geps - \s))^{1/2}-\id_{T\M}}{\e}.
\end{split}
\]
For matrices $A_\e$ converging to zero, we have the estimate
\[
|(\id_{T\M}+ A_\e)^{1/2} - \id_{T\M} - \tfrac12 \sym A_\e| \lesssim |A_\e|^2.
\]
Since $\geps\to s$ uniformly, we may use this estimate for $A_\e = \s^{-1} (\geps - \s)$ to obtain
\[
\left|\frac{\Qeps - \id_{T\M}}{\e} - \sym \s^{-1}\frac{\geps - \s}{2\e}\right| \lesssim |\geps - \s|\,\frac{|\geps - \s|}{\e}.
\]
Since the right-hand side is the product of a sequence converging to zero uniformly and a sequence bounded in $L^2$, it converges to zero in $L^2$, and hence
\[
\frac{\Qeps - \id_{T\M}}{\e}  \to \frac12 \sym \s^{-1}\h \equiv\xi
\qquad\text{in $L^2$}.
\]
Note that the symmetrization is coordinate-dependent (it is the symmetrization of the \emph{matrix} $\s^{-1}\h$), even though $\s^{-1}\h$ is a coordinate-independent section. 

It remains to establish the relation between $\xi$ and $\h$. Indeed, for vector fields $u,v\in\mathfrak{X}(\M)$,
\[
\begin{split}
\frac{(u,v)_\geps - (u,v)_\s}{\e} &= \frac{(\Qeps u,\Qeps v)_\s - (u,v)_\s}{\e} \\
&= \frac{(\Qeps u, \Qeps v - v)_\s}{\e}  + \frac{(\Qeps u - u, v)_\s}{\e} 
\end{split}
\]
The left-hand side converges in $L^2$ to $\h(u,v)$, whereas the right-hand side converges to $(u,\xi(v))_\s + (\xi(u),v)_\s$.
\end{proof}
%%%%%%%%%%%%%

With that, the following lemma completes the proof that a family of maps $\feps$ of bounded energies admits an $\e$-truncation:

%%%%%%%%%%%%%
\begin{lemma}
\label{lem:almost_the_same}
For every sequence $\feps \in W^{1,2}(\M;\S)$,
\[
\int_\M \dist^2(d\feps,\SO(\s,\feps^*\s))\,\VolS \lesssim \e^2\brk{\Eeps(\feps) + 1}.
\]
\end{lemma}
%%%%%%%%%%%%%

%%%%%%%%%%%%%
\begin{proof}
Since the left-hand side and $\e^2\,\Eeps(\feps)$ differ by distances and volume forms in the Riemannian manifolds $(\M,\geps)$ and $(\M,\s)$, we will denote those explicitly. We need to prove that
\[
\int_\M \dist_{\s,\feps^*\s}^2(d\feps,\SO(\s,\feps^*\s))\,\VolS \lesssim 
\int_\M \dist_{\geps,\feps^*\s}^2(d\feps,\SO(\geps,\feps^*\s))\,\VolEps + \e^2.
\]
Since the uniform convergence $\geps\to\s$ implies that distances and volume forms in $(\M,\geps)$ and $(\M,\s)$ are equivalent (uniformly in $\e$), it suffices to show that
\[
\int_\M \dist_{\s,\feps^*\s}^2(d\feps,\SO(\s,\feps^*\s))\,\VolS \lesssim 
\int_\M \dist_{\s,\feps^*\s}^2(d\feps,\SO(\geps,\feps^*\s))\,\VolS + \e^2.
\]
Let $V\in\SO(\s,\feps^*\s)$, and
let $\Qeps$ and $\xi$ be
as in Lemma~\ref{lem:Qeps}. 
Using this, we have
\[
\begin{split}
\dist_{\s,\feps^*\s}^2(d\feps,\SO(\s,\feps^*\s)) 
&\le |d\feps - V|_{\s,\feps^*\s}^2 \\
&\lesssim |d\feps -   V \,  \Qeps|_{\s,\feps^*\s}^2 + |V(\Qeps - \id_{T\M})|_{\s,\feps^*\s}^2 \\
&=  |d\feps -   V \,  \Qeps|_{\s,\feps^*\s}^2 + |\Qeps - \id_{T\M}|_{\s,\s}^2.
\end{split}
\]
Minimizing over $V$,
\[
\dist_{\s,\feps^*\s}^2(d\feps,\SO(\s,\feps^*\s)) \lesssim 
\dist_{\s,\feps^*\s}^2(d\feps,\SO(\geps,\feps^*\s)) + |\Qeps - \id_{T\M} |_{\s,\s}^2.
\]
Integrating over $\M$, using the fact that by \eqref{eq:Qeps} the second term is $O(\e^2)$,
we obtain the desired result.
\end{proof}
%%%%%%%%%%%%%

%%%%%%%%%%%%%%%%%%
\subsection{Convergence of maps and their displacements}

With this notion of $\e$-truncation, we define a notion of convergence of configurations with respect to which $\Gamma$-convergence will be defined. This is followed by establishing that the limit is a covariant derivative of a displacement field (Proposition~\ref{prop:J_is_a_gradient}), and that it is unique up to infinitesimal isometries of $(\S,\s)$, i.e., Killing vector fields (Proposition~\ref{prop:limit_unique}).
These vector fields do not affect the limit energy \eqref{eq:E_0}, and are thus immaterial.

%%%%%%%%%%%%%
\begin{definition}
\label{def:limit}
Let $\Psi\in\Iso(\S)$ be an isometry of $\S$ and let $J\in L^2\W^1(\M;T\M)$.
A sequence $\feps\in W^{1,2}(\M;\S)$ is said to converge to the pair $(\Psi,J)$, denoted by $\feps\to(\Psi,J)$, if there exists an $\e$-truncation $\hfeps \in W^{1,\infty}(\M;\S)$ of $\feps$ and a family $\Feps\in\Iso(\S)$, such that the following hold:
\begin{enumerate}[itemsep=0pt,label=(\alph*)]
\item $\Feps\to \Psi$.
\item $\|\kappa\circ \hfeps - \kappa\circ \Feps\|_{L^2} \lesssim \e$.
\item The following limit holds,
\[
\frac{d\Feps^{-1} \circ \Pi_{\hfeps}^{\Feps} \circ d\hfeps - \id_{T\M}}{\e} \weakly J
\qquad\text{in $L^2\W^1(\M;T\M)$},
\]
where $\Pi_{\hfeps}^{\Feps}: \hfeps^*T\M\to \Feps^*T\M$ is the parallel transport in $(T\S,\s)$ along the shortest geodesic connecting the images of $\hfeps$ and $\Feps$.
\end{enumerate}
If $J$ is continuous and the convergence in (c) is uniform, we say that $\feps$ converges to $(\Psi,J)$ uniformly.
\end{definition}
%%%%%%%%%%%%%

The object on the left-hand side of part (c) can be thought of as the generalized ``gradient'' of a displacement field (though it is not a true gradient, but, as we will show below, its limit is).

Note first that the isometry group $\Iso(\S)$ may be trivial. Moreover, since the isometry group of a compact manifold is a compact Lie group \cite[Thm.~3.4]{KN63}, every family $\Feps$ of isometries is guaranteed to have a converging subsequence.

For the parallel transport $\Pi_{\hfeps}^{\Feps}$ to be well-defined, the image of $\hfeps$ must reside within the injectivity radius of the image of $\Feps$ almost-everywhere in $\M$. 
Since $\hfeps$ is uniformly Lipschitz, it has a uniformly converging subsequence; by (a)--(b), it follows that the limit is $\Psi$, and hence $\hfeps$ converges to $\Psi$ uniformly.
Since $\Feps$ also converges uniformly to $\Psi$, it follows that for $\e$ small enough, the distance between $\hfeps(p)$ and $\Feps(p)$ will be less than the injectivity radius of $\S$ for all $p\in \M$, and hence $\Pi_{\hfeps}^{\Feps}$ is well-defined.

The next proposition shows that the limit (in the sense of Definition~\ref{def:limit}) of a sequence of configurations is always the $\nabla^\s$-covariant derivative of a vector field $u$ --- the limit displacement.

%%%%%%%%%%%%%
\begin{proposition}[Limit displacement]
\label{prop:J_is_a_gradient}
Let $\feps\to (\Psi,J)$ in the sense of Definition~\ref{def:limit}. 
Then there exists a vector field $u\in W^{1,2}\mathfrak{X}(\M)$, such that 
\[
J = \nabs u.
\]
Specifically, 
\beq
u(p) = \LimEps \frac{1}{\e} d\Feps^{-1} \exp^{-1}_{\Feps(p)} \feps(p),
\label{eq:explicit_u}
\eeq
where for $q\in\S$, $\exp_q:T_q\S\to\S$ is the exponential map and the limit is in $L^2$. 
\end{proposition}
%%%%%%%%%%%%%

%%%%%%%%%%%%%
\begin{proof}
Let $\hfeps\in W^{1,\infty}(\M;\S)$ and $\Feps\in\Iso(\S)$ be as in Definition~\ref{def:limit}, i.e., $\hfeps$ is an $\e$-truncation of $\feps$, $\Feps\to \Psi$, $\|\kappa\circ \hfeps - \kappa\circ \Feps\|_{L^2}\lesssim \e$, and 
\[
\frac{d\Feps^{-1}\circ \Pi_{\hfeps}^{\Feps} \circ d\hfeps - \id_{T\M}}{\e} \weakly J
\qquad\text{in $L^2\W^1(\M;T\M)$}.
\]
We start by observing that since $\Feps$ are isometries, $\Feps^{-1}\circ \hfeps$ are uniformly Lipschitz,
\[
\|\kappa\circ \hfeps - \kappa\circ \Feps\|_{L^2} \simeq \|\kappa\circ (\Feps^{-1}\circ \hfeps) - \kappa\|_{L^2},
\]
and
\beq\label{eq:Iso_and_par_trans}
d\Feps^{-1}\circ \Pi_{\hfeps}^{\Feps} \circ d\hfeps = \Pi_{(\Feps^{-1}\circ \hfeps) }^{\iota} \circ d(\Feps^{-1}\circ \hfeps) 
\eeq
where $\iota:\M\inj\S$ is the inclusion map. Thus, we may without loss of generality assume that $\Feps = \iota$ and that the $L^2$-convergence of $\hfeps$ to $\iota$ is at a rate $\e$. 
As pointed out above, this convergence is also uniform.

Since the claim is local, it can be restricted to a local coordinate chart. 
Let $V\subset \M$ be an open set, such that $\chi: W\to\R^d$ is a local coordinate system with $V\Subset W\subset \S$.
Since the sequence $\hfeps$ converges to the identity uniformly, $\hfeps(V) \subset W$ for small enough $\e$, so that we may consider the coordinate representations of $\hfeps$,
\[
F_\e = \chi\circ \hfeps \circ \chi^{-1} :\chi(V) \to \R^d.
\]

Define the $\R^d$-valued functions $U_\e:\chi(V)\to\R^d$, given by
\[
U_\e = \frac{F_\e - \id}{\e},
\]
where $\id: \chi(V) \to \R^d$ is the identity mapping.
Since all Riemannian metrics on a compact manifold are equivalent, distances  $d_\s(\cdot,\cdot)$ in $V$ are equivalent to Euclidean distances  in $\chi(V)\subset\R^d$,  which implies that
\[
|U_\e| = \frac{|F_\e - \id|}{\e} \simeq \frac{d_\s(\hfeps,\iota)}{\e} \circ \chi^{-1},
\]
and further that $\e U_\e\to0$ in $L^\infty$ and $\|U_\e\|_{L^2(\chi(V))} \lesssim 1$.
Moreover, $F_\e$ are uniformly Lipschitz.

We proceed to show that $U_\e$ is in fact bounded in $W^{1,2}$. 
We will use the following notations for paths in $W$ and their images in $\R^d$ along with their corresponding parallel transports. For $p\in V$, we denote by $\eta_\e(p)$ the shortest geodesic connecting $p$ and $\hfeps(p)$ and by $\Pi_{\eta_\e(p)} = \Pi_{\feps(p)}^p$ the parallel transport along $\eta_\e(p)$. 
Note that, by possibly making $V$ smaller, we can assume that $\eta_\e(p)\subset W$ for all $p\in V$ and all $\e$ small enough.
Likewise, we denote by $\gamma_\e(p)$ the curve in $W$ connecting $p$ and $\hfeps(p)$ along the coordinate segment,
\beq\label{eq:coordinate_line}
\gamma_\e(p) : t\mapsto \chi^{-1}\circ ((1-t)F_\e + t\,\id)\circ \chi \qquad t\in [0,1],
\eeq
and by $\Pi_{\gamma_\e(p)}$ the parallel transport along that coordinate segment.
Again, we can choose $W$ such that this coordinate segment lies indeed in $\chi(W)$.
With a slight abuse, we use the same notations for the coordinate representations of the curves and their corresponding parallel transports.

Write
\[
\nabla U_\e = \frac{\nabla F_\e - \id_{\R^d}}{\e} = 
\frac{\Pi_{\eta_\e}  \nabla F_\e - \id_{\R^d}}{\e} 
- \frac{\Pi_{\eta_\e}  - \Pi_{\gamma_\e}}{\e}  \nabla F_\e 
- \frac{\Pi_{\gamma_\e} - \id_{\R^d}}{\e}  \nabla F_\e,
\]
where $\nabla$ denotes the coordinate gradient.

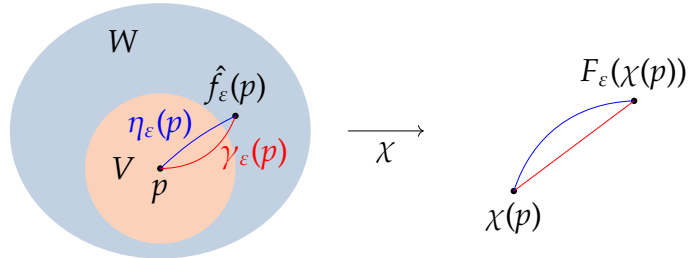
\[
\btkz
\fill[warmblue!30] (0,0.5) ellipse (2cm and 1.7cm);
\fill[ocre!30] (0,0) ellipse (1cm and 1cm);
\tkzDefPoint(0,0){p};
\tkzDefPoint(1,0.7){fp};
\tkzDrawPoints(p,fp);
\tkzLabelPoint[below](p){$p$};
\tkzLabelPoint[above](fp){$\hfeps(p)$};
\tkzDrawSegment[bend left=-35, color=red](p,fp);
\tkzDrawSegment[bend left=8,  color=blue](p,fp);
\tkzText(-0.5,1.7){$W$};
\tkzText(-0.5,0){$V$};
\tkzText[blue](-0.0,0.58){$\eta_\e(p)$};
\tkzText[red](1.25,0.20){$\gamma_\e(p)$};

\tkzDefPoint(2.5,0.5){a};
\tkzDefPoint(3.5,0.5){b};
\tkzDrawSegment[->](a,b);
\tkzLabelSegment(a,b){$\chi$};

\begin{scope}[xshift=5cm]
\tkzDefPoint(-0.3,-0.3){p};
\tkzDefPoint(1.3,0.9){fp};
\tkzDrawPoints(p,fp);
\tkzLabelPoint[below](p){$\chi(p)$};
\tkzLabelPoint[above](fp){$F_\e(\chi(p))$};
\tkzDrawSegment[color=red](p,fp);
\tkzDrawSegment[bend left=35,  color=blue](p,fp);
\end{scope}
\etkz
\]

The first term on the right-hand side is the coordinate representation of 
\[
\frac{\Pi_{\hfeps}^{\iota} \circ d\hfeps - \id_{T\M}}{\e},
\]
which weakly converges to $J$, and
hence is bounded in $L^2(\chi(V))$. Since the sequence $\nabla F_\e$ is bounded in $L^\infty$, the boundedness of $\nabla U_\e$ will follow from the following estimates:
\beq\label{eq:parallel_trans_estimates}
\frac{1}{\e}\|\Pi_{\eta_\e}  - \Pi_{\gamma_\e}\|_{L^2(\chi(V))} \to 0,
\Textand
\frac{1}{\e}\|\Pi_{\gamma_\e} - I\|_{L^2(\chi(V))} \lesssim 1.
\eeq

The first term in \eqref{eq:parallel_trans_estimates} can be written as
\beq\label{eq:holonomy1}
\frac{\Pi_{\eta_\e}  - \Pi_{\gamma_\e}}{\e} =  \frac{\Pi_{\sigma_\e}  - \id_{T\M}}{\e}\circ \Pi_{\gamma_\e},
\eeq
where $\sigma_\e(p)\subset W$ is the closed curve based in $p$, obtained by concatenating $\eta_\e(p)$ and $\gamma_\e(p)$.
For every $p\in V$, $\Pi_{\sigma_\e(p)}$ is the holonomy along the curve $\sigma_\e(p)$.
The holonomy estimate \cite[Proposition~B.3]{BMM23} implies
\beq\label{eq:holonomy_estimate}
|\Pi_{\sigma_\e(p)}  - \id_{T_p\M}| \lesssim \operatorname{len}_\s(\sigma_\e(p))^2,
\eeq
where $\operatorname{len}_\s(\sigma_\e(p))$ is the length of $\sigma_\e(p)$ (with respect to the metric $\s$), and the constant in the inequality depends only on properties of $(\S,\s)$.

Using again the equivalence of $d_\s$ and the Euclidean distance in the coordinate patch, and their equivalence to the extrinsic distances in $\kappa(\S)$, we obtain that the length of $\sigma_\e$ satisfies
\[
\operatorname{len}_\s(\sigma_\e(p)) \lesssim d_\s(\hfeps(p),p) \lesssim |\kappa \circ \hfeps(p) - \kappa(p)|
\]
where the inequality constants are independent of $p$.

Combining this with the holonomy estimate \eqref{eq:holonomy_estimate} and the fact that all parallel transports are uniformly bounded in norm, we obtain from \eqref{eq:holonomy1} that
\[
\Abs{\frac{\Pi_{\eta_\e}  - \Pi_{\gamma_\e}}{\e}} \lesssim \frac{1}{\e}|\kappa \circ \hfeps- \kappa|^2,.
\] 
Noting that $|\kappa \circ \hfeps- \kappa|$ tends to zero uniformly, and that $\frac{1}{\e}|\kappa \circ \hfeps- \kappa|$ is bounded in $L^2$, we obtain that the left-hand side tends to zero strongly in $L^2$.

As for the second term in \eqref{eq:parallel_trans_estimates}, for a point $x\in \chi(\M)$,
noting that $\gamma_\e(x)$ has velocity $\e\,U_\e(x)$,
the function $[0,1]\to \End(\R^d)$ given by
\[
g(t) = \Pi_{\gamma_\e(t)}^{\gamma_\e(0)} - I,
\]
satisfies the differential equation
\[
\dot{g}(t) = \e\, \Gamma_{\gamma_\e(t)}(U_\e,\cdot),
\qquad
g(0) = 0,
\]
where $\Gamma:\chi(V) \to \Hom(\R^d\otimes \R^d,\R^d)$ are the Christoffel symbols of the connection in the  coordinate system.
\footnote{Since we are considering the Levi-Civita connection, $\Gamma$ is symmetric; yet, the first and second variables have different meanings: The first is the direction of the curve along which we parallel transport, and the second is the vector we parallel transport. We retain this distinction below.}
By the mean-value theorem,
\beq
\frac{\Pi_{\gamma_\e(1)}^{\gamma_\e(0)} - I}{\e} = \frac{g(1) - g(0)}{\e} = \e^{-1}\, \dot g(\xi) =  \Gamma_{\gamma_\e(\xi)}(U_\e,\cdot)
\label{eq:use_me}
\eeq
for some $\xi\in(0,1)$ (which depends on $x$ and $\e$). Since the Christoffel symbols are uniformly bounded and $U_\e$ is  bounded in $L^2$, it follows that
\[
\frac{1}{\e}\left\| \Pi_{\gamma_\e} - I\right\|_{L^2(\chi(V))} \lesssim 1.
\]

This completes the proof of \eqref{eq:parallel_trans_estimates}, and thus $\nabla U_\e$ is bounded in $L^2$, and consequently $U_\e$ is bounded in $W^{1,2}$.
It follows that $U_\e$ has a subsequence converging weakly in $W^{1,2}$ to a limit $U$, which is the coordinate representation of a $W^{1,2}$-vector field $u$ on $V$. 
Note that it is not evident at this stage that the vector fields $u$ obtained in different coordinate patches glue together into a global field.

Denote by $\nabs U_\e$ the covariant derivative (in coordinates) of $U_\e$, i.e.,
\[
\nabs U_\e = \nabla U_\e + \Gamma(\cdot,U_\e),
\]
and consider the following sequence of maps $\chi(\M)\to \End(\R^d)$,
\[
H_\e = \frac{\Pi_{\eta_\e}  \nabla F_\e - I}{\e} - \nabs U_\e,
\]
which converge weakly in $L^2$ to the coordinate representation of $J - \nabs u$.  We will show that $J - \nabs u=0$ if we prove that $H_\e$ converges to zero in $L^1$. By a straightforward manipulation,
\[
\begin{split}
H_\e =
\frac{\Pi_{\eta_\e}  - \Pi_{\gamma_\e}}{\e}  \nabla F_\e +
\e\, \frac{\Pi_{\gamma_\e} - I}{\e}  \nabla U_\e +
\frac{\Pi_{\gamma_\e} - I}{\e}  -
\Gamma(\cdot,U_\e).
\end{split}
\]
By \eqref{eq:parallel_trans_estimates}, the first term tends to zero strongly in $L^2$. The second term is the product of $\e$ times two terms that are each bounded in $L^2$, and hence converges to zero in $L^1$. As for the last two terms, using \eqref{eq:use_me},
\[
\frac{\Pi_{\gamma_\e(x)} - I}{\e}  - \Gamma_x(\cdot,U_\e(x)) = 
\Gamma_{x + \xi\e U_\e(x)}(U_\e(x),\cdot) - \Gamma_x(\cdot,U_\e(x)),
\]
where $\xi\in(0,1)$ (which depends on $x$ and $\e$).
Since the connection is symmetric, $\Gamma_x$ is a symmetric bilinear form, and thus it follows from the mean value theorem that, for some $\eta\in (0,1)$,
\[
\Gamma_{x + \xi\e U_\e(x)}(U_\e(x),\cdot) - \Gamma_x(\cdot,U_\e(x)) = 
\e \xi \, \nabla\Gamma_{x + \eta\xi\e U_\e(x)}(U_\e(x), U_\e(x),\cdot),
\]
and the right-hand side tends to zero in $L^1$. 
This completes the proof that $J = \nabla^s u$ in the coordinate patch $V$. We will now prove the representation formula \eqref{eq:explicit_u} for $u$, which will show that it is indeed globally and intrinsically-defined.

Let $u_\e \in W^{1,2}\mathfrak{X}(V)$ be the vector fields whose local representations are $U_\e$. 
By the equivalence of metrics in $V$ and $\chi(V)$, the sequence $u_\e$ converges weakly in $W^{1,2}$ to the vector field $u$ whose local representation is $U$.
For $p\in V$, consider the intrinsic distance
\[
d_\s(\hfeps(p),\exp_p(\e\,u_\e(p)).
\] 
Note that in coordinates, for $x = \chi(p)$,
\[
|\exp_x(\e\,U_\e(x)) - x - \e\,U_\e(x)| \lesssim \e^2\,|U_\e(x)|^2,
\]
hence by that same equivalence of metrics,
\[
d_\s(\hfeps(p),\exp_p(\e\,u_\e(p)) \simeq |\exp_x(\e\,U_\e) - x - \e\,U_\e(x)| \lesssim \e^2\,|U_\e(x)|^2.
\]
On a compact manifold, the collection of maps $\exp_p:T_p\M\to S$, restricted to vectors of magnitude, say, half the injectivity radius, are uniformly bilipschitz. Hence,
\[
d_\s(\hfeps(p),\exp_p(\e\,u_\e(p)))  \simeq |\exp_p^{-1}\hfeps(p) - \e\,u_\e(p)|,
\] 
and thus we obtain that
\[
|\e^{-1} \exp_p^{-1}\hfeps(p) - u_\e(p)| \lesssim \e\,|U_\e(x)|^2.
\]
Since the right-hand side is the product of a sequence $\e\,|U_\e|$ tending to zero uniformly and a sequence $|U_\e|$ bounded in $L^2$, it converges to zero in $L^2$, which together with the fact that $u_\e \to u$ in $L^2$ implies that
\[
\{p\mapsto \e^{-1}\exp_p^{-1}\hfeps(p)\} \to u
\qquad\text{in $L^2(V)$},
\]
and therefore also globally in $L^2(\M)$.
\end{proof}
%%%%%%%%%%%%%

It remains to verify whether Definition~\ref{def:limit} defines a limit uniquely, which is not obvious as it hinges on the choice of two auxiliary sequences, $\hfeps$ and $\Feps$. As we will see, the choice of isometries $\Feps$ may affect limit up to an immaterial gradient of a Killing field. To show this, we first need a lemma regarding converging isometries:

%%%%%%%%
\begin{lemma}
\label{lem:isometries}
Let $\Feps\in\Iso(\S)$ satisfy
\beq
\|\kappa\circ\Feps - \kappa\|_{L^2} \lesssim \e.
\label{eq:iso_L2}
\eeq
Then, for every $k\in\bbN$,
\[
\|\kappa\circ\Feps - \kappa\|_{C^k} \lesssim \e,
\] 
where the constants depend on the constant in \eqref{eq:iso_L2} as well as on $k$.
Moreover, there exists a subsequence such that
\[
\frac{\Pi_{\Feps}^\iota\circ d\Feps - \id_{T\M}}{\e} \to \nabs u
\qquad
\text{in $C^k$}
\]
for some $u\in\iso(\S)$.
Here, the domain over which the norms are defined can be the whole $\S$ or any smooth subdomain $\M\subset \S$.
\end{lemma}
%%%%%%%

%%%%%%%
\begin{proof}
We start by showing that
\beq
\|\kappa\circ\Feps - \kappa\|_{C^0} \lesssim \e^\alpha,
\label{eq:bootstrap}
\eeq
for $\alpha = 2/(2+d)\in (0,1)$. Let $r_\e = \e^\alpha$, and note that for $q\in\S$
\[
\Vol_\s(B_{r_\e}(q)) \simeq r_\e^d,
\]
where, since the manifold is compact, the constants are independent of $q$. For every $q\in\S$ there exists a $q_\e$ satisfying
\[
d_\s(q_\e,q) \lesssim r_\e
\Textand
|\kappa\circ\Feps(q_\e) - \kappa(q_\e)| \lesssim \frac{\e}{\Vol_\s^{1/2}(B_{r_\e}(q))}, 
\]
for otherwise \eqref{eq:iso_L2} would be violated. By the triangle inequality, the fact that $\Feps$ are isometries, and the equivalence of extrinsic and intrinsic distances,
\[
\begin{split}
|\kappa\circ\Feps(q) - \kappa(q)| 
&\le |\kappa\circ\Feps(q) - \kappa\circ\Feps(q_\e)|  + |\kappa\circ\Feps(q_\e) - \kappa(q_\e)|  + |\kappa(q_\e) - \kappa(q)| \\
&\lesssim  r_\e + \frac{\e}{\Vol_\s^{1/2}(B_{r_\e}(q))} \simeq r_\e,
\end{split}
\]
and since this holds for every $q\in\S$, we obtain \eqref{eq:bootstrap}.

We proceed with a bootstrap to show that 
\beq
\|\kappa\circ\Feps - \kappa\|_{C^0} \lesssim \e^{\min(1,2\alpha)}.
\label{eq:bootstrap2}
\eeq
To this end, we write $\Feps$ in the form
\[
\Feps = \Fl_1(u_\e),
\]
where $u_\e\in\iso(\S)$ are Killing fields and $\Fl_1$ is the flow map up to time $1$. 
Since $\Feps$ converges to the identity map, $u_\e\to 0$, which holds in every norm, since $\iso(\S)$ is a finite-dimensional vector space.
Since $d\Fl_1(u_\e) = u_\e = d\exp(u_\e)$, the flow map and the exponential map identify up to linear order in the vector field, and hence we have, for every $p\in\S$,
\[
d_\s(\Feps(p),\exp_p(u_\e(p))) \lesssim  \|u_\e\|^2_{C^0}.
\]
Hence by the triangle inequality,
\[
|u_\e(p)| = d_\s(\exp_p(u_\e(p)), p) \lesssim d_\s(\Feps(p),p) + \|u_\e\|^2_{C^0}.
\]
Using the fact that all norms on $\iso(\S)$ are equivalent, we obtain that 
\[
\|u_\e\|_{C^0} \lesssim \|u_\e\|_{L^2} \lesssim \e^{\min(1,2\alpha)},
\]
and in turn that
\[
\|d_\s(\Feps(p),p) \|_{C^0} \lesssim  \|u_\e\|_{C^0} \lesssim \e^{\min(1,2\alpha)},
\]
which implies \eqref{eq:bootstrap2}.

Repeating this last step finitely-many times, we obtain that
\[
\|u_\e\|_{L^2} \lesssim \e.
\]
Using again the fact that $\iso(\S)$ is finite-dimensional, we have for every $k\in\bbN$, 
\[
\|u_\e\|_{C^k} \lesssim \e,
\]
where the constant depends on $k$. By the stability of solutions of ordinary differential equations with respect to perturbations in the flow fields, it follows that in every coordinate patch,
\[
\|\Feps - \iota\|_{C^k} \lesssim \e,
\]
and thus that globally,
\[
\|\kappa\circ\Feps - \kappa\|_{C^k} \lesssim \e.
\] 

As the for the second statement, we note that the $C^1$ convergence above implies that $\Pi_{\Feps}^\iota\circ d\Feps - \id_{T\M} = O(\e)$ pointwise, and hence there exists a subsequence of $\Feps$ that converges in the sense of Definition~\ref{def:limit}.
By repeating the analysis of Proposition~\ref{prop:J_is_a_gradient} in a coordinate patch, we obtain that 
\[
\frac{\Pi_{\Feps}^\iota\circ d\Feps - \id_{T\M}}{\e} \to \nabs u
\qquad
\text{in $C^k$}
\]
where $u$ is the limit of the vector fields $u_\e$, and hence a Killing field.
\end{proof}
%%%%%%

%%%%%%%%%%%%%
\begin{proposition}[Uniqueness of limit]
\label{prop:limit_unique}
Suppose that both $\feps\to (\Psi,\nabs u)$ and $\feps\to (\Psi',\nabs u')$ in the sense of Definition~\ref{def:limit}. Then 
$\Psi=\Psi'$ and $u - u'$ belongs to the (possibly trivial) Lie algebra $\iso(\S)$ of Killing fields.
\end{proposition}
%%%%%%%%%%%%%

%%%%%%%
\begin{proof}
Assume that $(\hfeps,\Feps)$ and $(\vfeps,\vFeps)$ are the associated truncations and isometries yielding the limits $(\Psi,\nabs u)$ and $(\Psi',\nabla u')$, respectively.
From the definition of the $\e$-truncation and the convergence, 
\[
\begin{split}
\|\kappa \circ \Feps -\kappa \circ \vFeps\|_{L^2} &\le \|\kappa \circ \Feps -\kappa \circ \hfeps\|_{L^2} + \|\kappa \circ \vFeps -\kappa \circ \vfeps\|_{L^2} \\
&+ \|\kappa \circ \feps -\kappa \circ \hfeps\|_{L^2} + \|\kappa \circ \feps -\kappa \circ \vfeps\|_{L^2} \lesssim \e,
\end{split}
\]
hence $\Psi = \Psi'$, and $\|\kappa\circ (\Feps^{-1} \circ\vFeps) - \kappa\|_{L^2}\lesssim \e$.

By definition,
\beq\label{eq:uniqueness_aux1}
\frac{d\Feps^{-1}\circ \Pi_{\hfeps}^{\Feps} \circ d\hfeps - d{\vFeps}^{-1}\circ \Pi_{\vfeps}^{\vFeps} \circ d\vfeps}{\e} \weakly \nabs(u - u')
\qquad\text{in $L^2\W^1(\M;T\M)$}.
\eeq
We rewrite the left-hand side, using \eqref{eq:Iso_and_par_trans}, as follows,
\[
\begin{aligned}
&
\frac{\id_{T\M} -
\Pi_{\vFeps^{-1}\circ\Feps}^{\iota}\circ d(\vFeps^{-1}\circ  \Feps)}{\e} \circ d\Feps^{-1} \circ \Pi_{\hfeps}^{\Feps} \circ d\hfeps \\ 
&\quad+ d\vFeps^{-1}\circ \frac{\Pi_{\Feps}^{\vFeps} \circ \Pi_{\hfeps}^{\Feps} -
\Pi_{\hfeps}^{\vFeps}}{\e} \circ d\hfeps
+
d\vFeps^{-1}\circ\frac{\Pi_{\hfeps}^{\vFeps} \circ d\hfeps -
\Pi_{\vfeps}^{\vFeps} \circ d\vfeps}{\e}
\end{aligned}
\]
The first term in the product of a sequence converging (by Lemma~\ref{lem:isometries}, in any $C^k$ norm) to the gradient of a Killing field and a sequence converging strongly in $L^2$ to the identity, hence converges in $L^2$ to the gradient of a Killing field (the  compositions in the equation above are compositions of linear operators, so for the sake of convergence, or, equivalently, in coordinates, they behave like products). 
The second term is a product of bounded sequences and (by the same argument as in the proof of Proposition~\ref{eq:parallel_trans_estimates}) a holonomy term that tends to zero in $L^2$. 
Finally, the third term is bounded in $L^2$, by \eqref{eq:uniqueness_aux1} and the boundedness of the first two terms. Furthermore, it converges to zero in measure, since $\hfeps$ and $\vfeps$ identify over increasing sets; hence it converges to zero weakly in $L^2$. 
This completes the proof. 
\end{proof}
%%%%%%

%%%%%%%%%%%%%%%%%%%%%%%%%%%%%%%%%%%%%%%%%%%%%
\section{$\Gamma$-convergence}

%%%%%%%%%%%%%
\begin{theorem}[$\Gamma$-convergence]\label{thm:Gamma_conv}
The sequence of energies $E_\e :W^{1,2}(\M;\S) \to \R$, defined in \eqref{eq:Eeps}, $\Gamma$-converges, under the notion of convergence $\feps\to (\Psi,\nabs u)$ of Definition~\ref{def:limit}, to the limit energy functional, $E_0 : W^{1,2}\mathfrak{X}(\M)\to\R$ given by
\[
E_0(u) =  \frac{1}{4} \int_\M |\mathcal{L}_u\s  - \h|^2\,  \VolS,
\]
where $\mathcal{L}_u\s$ is given by \eqref{eq:calL},
and the norm in the integrand is the norm on $T^*\M\otimes T^*\M$ induced by $\s$.
\end{theorem}
%%%%%%%%%%%%%

The proof of Theorem~\ref{thm:Gamma_conv} will follow from Proposition~\ref{prop:lsc} (lower-semicontinuity) and Proposition~\ref{prop:recovery} (recovery sequence) below.

We will need the following lemma, which will be used when expanding the integrand in $\Eeps$ for $d\feps$ close to an isometry:

%%%%%%%%%%%%%
\begin{lemma}
\label{lem:SO}
Let $(V_1,\g_1)$ and $(V_2,\g_2)$ be oriented inner-product spaces of the same dimension and let $W:\Hom(V_1,V_2)\to\R$ be given by
\[
W(A) = \dist^2(A,\SO(\g_1,\g_2)).
\]
For $Q\in \SO(\g_1,\g_2) \subset \Hom(V_1,V_2)$,
\[
DW_Q(B) = 0
\Textand
D^2W_Q(B,B) = 2\, |\sym BQ^T|^2.
\]
Here $Q^T\in\SO(\g_2,\g_1) \subset \Hom(V_2,V_1)$ is the adjoint of $Q$, 
i.e., for $u\in V_2$ and $v\in V_1$,
\[
(Q^T(u),v)_{\g_1} = (u,Q(v))_{\g_2},
\]
and hence $BQ^T\in \End(V_2)$, which can be symmetrized with respect to $\g_2$, namely, $2\,\sym BQ^T = BQ^T + QB^T$
(see the Notations section in the Introduction for the definitions of symmetrization).
\end{lemma}
%%%%%%%%%%%%%

%%%%%%%%%%%%%
\begin{proof}
If $A$ is an orientation-preserving isomorphism, then \cite{Kah11}
\[
W(A) = |A - A(A^TA)^{-1/2}|_{\Hom(V_1,V_2)}^2.
\]
Differentiating once, 
\[
D_AW(B) = 2 \brk{A - A(A^TA)^{-1/2}, B - B(A^TA)^{-1/2} - \tfrac12 A (A^TA)^{-1/2}A^{-1} B - \tfrac12 AB^T A^{-T} (A^TA)^{-1/2}}_{\Hom(V_1,V_2)},
\]
thus obtaining that $D_QW(B)=0$ for $Q\in \SO(\g_1,\g_2)$, since $Q^T Q=I_{V_1}$. Differentiating once more and evaluating at $Q\in \SO(\g_1,\g_2)$,
\[
D^2_QW(B,B) = 2\, |B - B - \tfrac12 B - \tfrac12 Q B^T Q^{-T}|_{\Hom(V_1,V_2)}^2 = 
\tfrac{1}{2} |B Q^T  + Q B^T|_{\End(V_2)}^2 = 
2\, |\sym BQ^T|_{\End(V_2)}^2,
\]
where in the second passage we used the fact that $Q$ is an isometry.
\end{proof}
%%%%%%%%%%%%%

%%%%%%%%%%%%%
\begin{proposition}[Lower-semicontinuity]
\label{prop:lsc}
Let $\feps\to (\Psi,\nabs u)$ in the sense of Definition~\ref{def:limit}, where $\Psi\in\Iso(\S)$ and $u \in W^{1,2}\mathfrak{X}(\M)$. Then,
\[
E_0(u) \le \LiminfEps \Eeps(\feps).
\]
\end{proposition}
%%%%%%%%%%%%%

%%%%%%%%%%%%%
\begin{proof}
Let $\hfeps\in W^{1,\infty}(\M;\S)$ and $\Feps\in\Iso(\S)$ be as in Definition~\ref{def:limit},
and let $\Qeps$ and $\xi$ be defined as in Lemma~\ref{lem:Qeps}.  
Let $\eta_\e$ be the indicator functions of the sets $\{\feps = \hfeps\}$, which converges to $1$ boundedly in measure.
Now,
\[
\begin{split}
\Eeps(\feps) &\ge \frac{1}{\e^2} \int_\M \eta_\e\, \dist^2(d\feps,\SO(\geps,\feps^*\s))\, \VolEps \\
&= \frac{1}{\e^2} \int_\M \eta_\e\, \dist^2(d\hfeps,\SO(\geps,\hfeps^*\s))\, \VolEps  \\
&= \frac{1}{\e^2} \int_\M \eta_\e\, \dist^2(d\Feps^{-1}\circ \Pi_{\hfeps}^{\Feps}\circ d\hfeps,\SO(\geps,\s))\, \VolEps  \\
&= \frac{1}{\e^2} \int_\M \eta_\e\, \dist^2(\Qeps + \e\,\Jeps,\SO(\geps,\s))\, \VolEps ,
\end{split}
\]
where 
\[
\Jeps = \frac{d\Feps^{-1}\circ \Pi_{\hfeps}^{\Feps}\circ d\hfeps - \id_{T\M}}{\e} +  \frac{\id_{T\M} - \Qeps}{\e}.
\]
The passage to the third line follows from parallel transport and $d\Feps$ being linear isometries of the corresponding inner-product spaces.
By \eqref{eq:Qeps} and Definition~\ref{def:limit},
\[
\Jeps \weakly \nabs u  - \xi
\qquad\text{in $L^2\W^1(\M;T\M)$}.
\]
The rest of the proof is standard: 
Let $\chi_\e$ be the indicator function of the set $\{|\Jeps|<\e^{-1/2}\}$; since $\Jeps$ is uniformly bounded in $L^2$, the function $\chi_\e$ converges to 1 boundedly in measure. Expanding to quadratic order, using Lemma~\ref{lem:SO},
\[
\begin{split}
\Eeps(\feps) &\ge \frac{1}{\e^2} \int_\M \eta_\e \chi_\e \dist^2(\Qeps + \e \Jeps,\SO(\geps,\s))\, \VolEps \\
&= \int_\M \eta_\e \chi_\e \brk{|\sym\Jeps \Qeps^T|^2 + |\Jeps|^2 \frac{\omega(\e \,\Jeps)}{\e^2 |\Jeps|^2}} \frac{\VolEps}{\VolS}\VolS,
\end{split}
\]
where $\omega(A)/|A|^2\to 0$ as $|A|\to0$. 
Since the product of an $L^2$-weakly converging sequence and a sequence converging in measure boundedly converges weakly in $L^2$ to the product of the limits, and since $\Qeps\to\id_{T\M}$ and $\frac{\VolS}{\VolEps} \to 1$ uniformly (the latter due to the uniform convergence $\geps\to \s$),
\[
\brk{\frac{\VolEps}{\VolS}}^{1/2} \eta_\e \chi_\e \Jeps \Qeps^T \weakly  \nabs u  - \xi \quad \text{in $L^2$}.
\]
Letting $\e\to0$, using the $L^2$-boundedness of $\Jeps$ and the lower-semicontinuity of quadratic forms,
\[
\liminf_{\e\to0} \Eeps(\feps)  \ge \int_\M |\sym(\nabs u  - \xi)|^2\,\VolS.
\]
By applying a musical isomorphism $\flat: T^*\M\otimes T\M \to T^*\M\otimes T^*\M$, 
along with \eqref{eq:calL} and \eqref{eq:xi_h}, we obtain the desired result.
\end{proof}
%%%%%%%%%%%%%

The following lemma will be used in the construction of recovery sequences:

%%%%%%%%%%%%%
\begin{lemma}
\label{lem:DOC92}
Let $u \in \mathfrak{X}(\overline{\M})$, let $\Psi\in\Iso(\S)$, and let $\feps\in C^\infty(\overline{\M};\S)$ be defined by
\[
\feps(p) = \exp_{\Psi(p)}(\e\, d\Psi(u(p))).
\] 
Then, $\feps\to(\Psi,\nabs u)$ uniformly in the sense of Definition~\ref{def:limit}.
\end{lemma}
%%%%%%%%%%%%%

%%%%%%%%%%%%%
\begin{proof}
Since the exponential map and parallel transport are invariant under isometries, it suffices to prove the claim for $\Psi = \iota$.
First, $\|\kappa\circ \feps - \kappa\|_{L^2} \lesssim \e$ because $d_\s(p,\exp_p(v)) = |v|$, from which it also follows that the parallel transport from $\feps(p)$ to $p$ is well-defined for $\e$ small enough.
Let $p\in \M$, let $X\in T_p\M$ and let $\gamma$ be a curve in $\M$ satisfying $\gamma(0) = p$ and $\dot{\gamma}(0) = X$. 
Consider the parametrized surface in $\M$ given by 
\[
B(\e,t) = \feps(\gamma(t)) = \exp_{\gamma(t)}(\e\,  u (\gamma(t)).
\]
First,
\[
\pdLim{\e} B(\e,t) =  u (\gamma(t)),
\]
which is a vector field along $\gamma$. Taking the $\s$-covariant derivative with respect to $t$,
\[
\DerivLim{t} \pdLim{\e} B(\e,t) = (\nabs_X u)_p.
\]
Second,
\[
\pdLim{t} B(\e,t) = d\feps(X),
\]
which is a vector field along the curve $\e\to \feps(p)$.
By \cite[p. 56, Ex.~2]{DOC92},
\[
\DerivLim{\e} \pdLim{t} B(\e,t) = \LimEps \frac{\Pi_{\feps}^\iota \circ d\feps(X) - X}{\e}
\]
Finally, by the symmetry of the connection, these mixed derivatives are equal \cite[p. 68, Lemma 3.4]{DOC92}, 
which yields 
that
\[
 \frac{\Pi_{\feps}^\iota \circ d\feps - \id_{T\M}}{\e} \to \nabs u
\]
pointwise.
To verify that the convergence is in fact uniform, we may perform this analysis in local coordinate patches, as in Proposition~\ref{prop:J_is_a_gradient}.
\end{proof}
%%%%%%%%%%%%%

%%%%%%%%%%%%%
\begin{proposition}[Recovery sequence]
\label{prop:recovery}
Let $u \in W^{1,2}\mathfrak{X}(\M)$ and let $\Psi\in\Iso(\S)$. Then there exists a sequence $\feps\in W^{1,2}(\M;\S)$ converging to $(\Psi,\nabs u)$ in the sense of Definition~\ref{def:limit}, such that
\[
E_0(u) = \LimEps \Eeps(\feps).
\]
\end{proposition}
%%%%%%%%%%%%%

%%%%%%%%%%%%%
\begin{proof}
Suppose first that $u $ is smooth up to the boundary. 
Set $\Feps = \Psi$ and define for every $p\in\M$,
\[
\feps(p) = \exp_{\Psi(p)}(\e\, d\Psi(u(p))).
\]
Since $u\in \mathfrak{X}(\overline{\M})$, there is not need to introduce an $\e$-truncation; we note that even though the Lipschitz constants of $\feps$ depend on $u$, for every $u$ there exists an $\e_1(u)$ such that, say, $\operatorname{Lip} \feps < 2$ for every $\e < \e_1(u)$.
By Lemma~\ref{lem:DOC92}, $\feps \to (\Psi,\nabs u)$. 
We then proceed as in the proof of the lower-semicontinuity, obtaining first
\[
\Eeps(\feps) = \frac{1}{\e^2} \int_\M \dist^2(\Qeps + \e\,\Jeps,\SO(\geps,\s))\, \VolEps ,
\]
where
\[
\Jeps = \frac{d\Psi^{-1}\circ \Pi_{\feps}^{\Psi}\circ d\feps - \id_{T\M}}{\e} +  \frac{\id_{T\M} - \Qeps}{\e}.
\]
Note that by Lemma~\ref{lem:DOC92},  
\[
\Jeps \to \nabs u  - \xi
\qquad\text{in $L^2\W^1(\M;T\M)$},
\]
i.e., the convergence is strong rather than weak.

Since we need to obtain an equality, we need to work more delicately with 
the indicator function $\chi_\e$ of the set $\{|\Jeps|<\e^{-1/2}\}$, obtaining 
\[
\begin{split}
\Eeps(\feps) &= 
\int_\M  \chi_\e \brk{|\sym\Jeps \Qeps^T|^2 + |\Jeps|^2 \frac{\omega(\e \,\Jeps)}{\e^2 |\Jeps|^2}} \frac{\VolEps}{\VolS}\VolS \\
&+ 
\frac{1}{\e^2} \int_\M  (1-\chi_\e) \, \dist^2(\Qeps + \e \Jeps,\SO(\geps,\s))\, \VolEps,
\end{split}
\]
where $\omega(A)/|A|^2\to 0$ as $|A|\to0$. The first term tends to $E_0(u)$, using the same arguments as in Proposition~\ref{prop:lsc}, and the fact that  quadratic forms are continuous with respect to strong $L^2$-convergence.
As for the second term, consider the integrand,
\[
\frac{1}{\e^2} (1-\chi_\e) \, \dist^2(\Qeps + \e \Jeps,\SO(\geps,\s)) \lesssim (1-\chi_\e) |\Jeps|^2.
\]
Since $\Jeps$ converges strongly in $L^2$, $|\Jeps|^2$ is uniformly-integrable. The product of a uniformly-integrable sequences and a sequence converging to zero boundedly in measure converges to zero in $L^1$, and hence the second term tends to zero. We have thus proved the theorem when $u$ is smooth.

Now, let  $u \in W^{1,2}\mathfrak{X}(\M)$ and let $u_n\in\mathfrak{X}(\overline{\M})$ be a sequence converging in $W^{1,2}$ to $u$.  By the $W^{1,2}$-continuity of the limit energy,
\[
\limn E_0(u_n) = E_0(u).
\]
By the first part of the proof, since $u_n$ are smooth, there exists   for every $n$ a sequence $\feps^n\in W^{1,\infty}(\M;\S)$ of uniformly-Lipschitz maps, such that $\feps^n \to (\Psi,\nabs u_n)$ and 
\[
\LimEps \Eeps(\feps^n) = E_0(u_n).
\]
It remains to invoke a diagonal argument to produce a sequence $\feps$ satisfying the desired properties.
First, as long as we take $\feps^n$ with $\e< \e_1(n) := \e_1(u_n)$, we do not need to truncate the sequence. 
Likewise, there exists for every $n$ an $\e_2(n)$, such that
\[
|\Eeps(\feps^n) - E_0(u_n)| < \frac{1}{n}
\]
for every $\e < \e_2(n)$. 
Finally, since the weak $L^2$-topology is metrizable on bounded sets, denoting by $\varrho$ such a metric on the $L^2$-ball of radius $2\|\nabs u\|_{L^2}$, there exists an $\e_3(n)$, such that 
\[
\varrho\brk{\frac{d\Psi^{-1} \circ \Pi_{\feps^n}^\Psi \circ d\feps^n- \id_{T\M}}{\e}, \nabs u_n} < \frac{1}{n}
\]
for every $\e < \e_3(n)$ (the fact that the first term in the parentheses also lies within this ball on which $\varrho$ is defined follows from the strong convergence $\feps^n\to (\Psi,\nabs u_n)$).
Thus, we may produce a function $\e_0(n) < \min\BRK{\e_1(n),\e_2(n),\e_2(n)}$, which without loss of generality can be assumed to be decreasing to zero. For every $\e < \e_0(1)$,  define $n(\e)$ such that
\[
\e \in (\e_0(n(\e)+1), \e_0(n(\e))). 
\]
By construction, $n(\e)\to\infty$ as  $\e\to0$. 
Define then,
\[
\feps = \feps^{n(\e)}.
\]
Since for every $\e < \e_0(1)$ it holds that $\e < n_0(\e)$, it follows that $\operatorname{Lip} \feps < 2$,
\[
|\Eeps(\feps) - E_0(u_{n(\e)})| < \frac{1}{n(\e)},
\]
and
\[
\varrho\brk{\frac{d\Psi^{-1} \circ \Pi_{\feps}^\Psi \circ d\feps- \id_{T\M}}{\e}, \nabs u_{n(\e)}} < \frac{1}{n(\e)}.
\]
Letting $\e\to0$, since $\varrho(\nabs u_{n(\e)},\nabs u) \to 0$, 
\[
\LimEps \varrho\brk{\frac{d\Psi^{-1} \circ \Pi_{\feps}^\Psi \circ d\feps- \id_{T\M}}{\e}, \nabs u} = 0,
\]
i.e., $\feps\to(\Psi,\nabs u)$,
and
\[
\LimEps \Eeps(\feps) = E_0(u),
\]
which completes the proof.
\end{proof}
%%%%%%%%%%%%%

%%%%%%%%%%%%%%%%%%%%%%%%%%%%%%%%%%%%%%%%%%%%%
\section{Compactness}

The last component for obtaining a limit model for weakly-incompatible elasticity is compactness:  showing that if $\feps$ is a sequence of configurations satisfying  $\Eeps(\feps) \lesssim 1$, then, modulo a subsequence, $\feps$ converges in the sense of Definition~\ref{def:limit} to $(\Psi,\nabs u)$ for some $\Psi\in\Iso(\S)$ and $u\in W^{1,2}\mathfrak{X}(\M)$. We need essentially to show the existence of an $\e$-truncation $\hfeps\in W^{1,\infty}(\M;\S)$ of $\feps$ and a sequence $\Feps$ of isometries, such that $\|\kappa\circ \hfeps - \kappa\circ \Feps\|_{L^2} \lesssim \e$, and
\[
\frac{d\Feps^{-1}\circ \Pi_{\hfeps}^{\Feps}\circ d\hfeps - \id_{T\M}}{\e}
\qquad
\text{is bounded in $L^2\W^1(\M;T\M)$}.
\]  

As exposed in the Introduction, this compactness property follows if $(\S,\s)$ satisfies a quantitative rigidity property as in Definition~\ref{def:FJM}:

%%%%%%%%%%
\begin{proposition}
\label{prop:3.1}
Suppose that $(\S,\s)$ satisfies a quantitative rigidity property.
Let $\feps\in W^{1,2}(\M;\S)$ satisfy $\Eeps(\feps)\lesssim 1$. 
Then, there exists an isometry $\Psi\in\Iso(\S)$ and a vector field $u\in W^{1,2}\mathfrak{X}(\M)$, such that $\feps$ has a subsequence converging to $(\Psi,\nabs u)$ in the sense of Definition~\ref{def:limit}.
\end{proposition}
%%%%%%%%%

%%%%%%%
\begin{proof}
By Definition~\ref{def:FJM} and Lemma~\ref{lem:almost_the_same} there exists for every $\feps$ an isometry $\Feps\in\Iso(\S)$, such that
\[
\|\kappa\circ \feps - \kappa\circ \Feps\|_{W^{1,2}(\M;\R^{D})} \lesssim \e.
\]  
Let $\hfeps\in W^{1,\infty}(\M;\S)$ be an $\e$-truncation of $\feps$. Since $\feps$ and $\hfeps$ are $O(\e)$-close in $W^{1,2}$, it follows that
\[
\|\kappa\circ \hfeps - \kappa\circ \Feps\|_{W^{1,2}(\M;\R^D)} \lesssim \e,
\]
and in particular, this holds for the $L^2$-norm.
By the compactness of $\Iso(\S)$, we may move to a subsequence such that $\Feps\to \Psi$.  
Since $\kappa$ is an isometric immersion,
\[
\begin{split}
|d\Feps^{-1}\circ \Pi_{\hfeps}^{\Feps} \circ d\hfeps - \id_{T\M}|^2 
&= |\Pi_{\hfeps}^{\Feps} \circ d\hfeps - d\Feps |^2 \\
&= |d\kappa\circ \Pi_{\hfeps}^{\Feps} \circ d\hfeps - d\kappa\circ d\Feps |^2 \\
&\lesssim |(d\kappa\circ \Pi_{\hfeps}^{\Feps} - d\kappa) \circ d\hfeps|^2 +
|\kappa\circ d\hfeps - d\kappa\circ d\Feps |^2.
\end{split}
\]
Recall that the various $d\kappa$ are evaluated at different points and therefore cannot be factored out.
Since $\kappa$ is smooth, $\hfeps$ and $\Feps$ are uniformly close, and since $|\Pi_{p}^{q}| = d_\s(p,q) + O(d_\s(p,q)^2)$,
\[
|d\kappa\circ \Pi_{\hfeps}^{\Feps} - d\kappa| \lesssim |\kappa\circ \hfeps - \kappa\circ \Feps|.
\]
Since $d\hfeps$ is uniformly bounded, we obtain upon integration that
\[
\|d\Feps^{-1}\circ \Pi_{\hfeps}^{\Feps} \circ d\hfeps - \id_{T\M}\|^2_{L^2(\M;\End(T\M))} \lesssim  \|\kappa\circ \hfeps - \kappa\circ \Feps\|^2_{W^{1,2}(\M;\R^D)},
\]
which implies that
\[
\frac{d\Feps^{-1}\circ \Pi_{\hfeps}^{\Feps} \circ d\hfeps - \id_{T\M}}{\e}
\]
is a bounded sequence in $L^2\W^1(\M;T\M)$, and hence has a weakly-converging subsequence, with limit $J$. 
By Proposition~\ref{prop:J_is_a_gradient}, $J = \nabs u$ for some $u\in W^{1,2}\mathfrak{X}(\M)$.
\end{proof}
%%%%%%

The following theorem proves that spheres satisfy a quantitative rigidity property. 
Note however, that the proof hinges on the Euclidean FJM theorem, and the technique does not seem to be generalizable to any other non-Euclidean manifold:

%%%%%%%%%
\begin{theorem}
\label{thm:AKM25b}
Let $p\in (1,\infty)$. Let $(\S,\s)$ be a round $d$-sphere and denote by $\kappa:\S\to\R^{d+1}$ the standard isometric embedding. Let $\M\subset\S$ be a Lipschitz domain. There exists for every $f\in W^{1,p}(\M;\S)$ an isometry $\Psi$ of $\S$, such that
\[
\|\kappa\circ f - \kappa\circ \Psi\|_{W^{1,p}(\M;\R^{d+1})}^2 \lesssim \|\dist(df,\SO(\s,f^*\s))\|_{L^p(\M)}^2.
\]  
\end{theorem}
%%%%%%%%

%The same proof holds in any $W^{1,p}$-regularity, $p\in (1,\infty)$.

%%%%%%%%%%%%%
\begin{proof}
Without loss of generality, we can assume that $(\S,\s)$ is a unit sphere.
The idea is to ``thicken" the $d$-sphere into an open subset of $\R^{d+1}$, endowed with the Euclidean metric $\euc$.
Consider the following diagram:
\[
\begin{xy}
(-10,0)*+{(\bbS^d,\s)}="S";
(30,0)*+{(\bbS^d\times(1/2,2),\tilde\s)}="RS";
(30,20)*+{(\R^{d+1},\euc)}="R";
{\ar@{->}_{\zeta}"S";"RS"};
{\ar@{->}^{\kappa}"S";"R"};
{\ar@{->}_{\Phi(p,r) = r\,\kappa(p)}"RS";"R"};
\end{xy}
\]
The embedding $\zeta:\bbS^d\inj \bbS^d\times(1/2,2)$ is given by
\[
\zeta(p) = (p,1),
\]
and the metric $\S$ on $\bbS^d\times(1/2,2)$ is
\[
\tilde\s_{(p,r)} =  r^2\,  \s + dr \otimes dr.
\]
By construction, $\kappa:(\bbS^d,\s)\to(\R^{d+1},\euc)$, $\zeta:\bbS^d\to\bbS^d\times(1/2,2)$ and $\Phi:(\bbS^d\times(1/2,2),\tilde\s)\to(\R^{d+1},\euc)$ are  isometric embeddings, the latter into an ambient space of the same dimension. Furthermore, there exists a one-to-one correspondence between isometries in these spaces: to every isometry $\hPsi\in\SO(d+1) \subset\Iso(\R^{d+1})$ corresponds a unique isometry $\Psi \in\Iso(\bbS^d)$, such that
\[
\hPsi\circ \kappa = \kappa \circ \Psi,
\]
or equivalently,
\[
\hPsi \circ\Phi(p,r) = \Phi(\Psi(p),r) = r \, \kappa(\Psi(p))
\qquad
(p,r) \in \bbS^d\times(1/2,2).
\]

Given $\M\subset\bbS^d$, denote
\[
\tilde{\M} = \M \times (1/2,2)\subset \bbS^d\times(1/2,2)
\Textand
\W = \Phi(\tilde{\M}) \subset\R^{d+1}.
\]
Let $f\in W^{1,p}(\M; \bbS^d)$, and consider the ```thickened" map $\hf\in W^{1,p}(\tilde{\M};\bbS^d\times(1/2,2))$ given by
\[
\hf(p,r) = (f(p),r),
\]
and the map
\[
\Phi\circ \hf\circ \Phi^{-1} \in W^{1,p}(\W;\R^{d+1}).
\]
By the FJM theorem \cite{FJM02b} for maps $\W\to\R^{d+1}$, there exists an isometry $\hPsi\in\SO(d+1)$, such that
\[
\|d(\Phi\circ \hf\circ \Phi^{-1}) - \hPsi\|_{L^p(\W;\R^{d+1})} \le 
C\, \|\dist(d(\Phi\circ \hf\circ \Phi^{-1}),\SO(d+1))\|_{L^p(\W;\R^{d+1})},
\]
where $C = C(\M,p)$. Since $\Phi$ is an isometric immersion, we may change variables, obtaining
\beq
\|d(\Phi\circ \hf) - d(\hPsi\circ \Phi)\|_{L^p(\tilde{\M};\R^{d+1})} \le 
C\, \|\dist(d(\Phi\circ \hf),\SO(\S,\euc))\|_{L^p(\tilde{\M};\R^{d+1})}.
\label{eq:FJMcodim1}
\eeq
Since $\Phi\circ \hf(p,r) = r\,\kappa(f(p))$, it follows that
\[
d(\Phi\circ \hf) = (\kappa\circ f)\, dr + r\, d\kappa\circ df.
\]
We now show  that
\[
(\kappa\circ f)\, dr + r\, d\kappa\circ \SO(df) \in \SO(\S,\euc),
\]
where $\SO(df)$ is a projection of $df$ onto $\SO(f^*\s,\s)$.
Indeed, denote
\[
A = (\kappa\circ f)\, dr + r\, d\kappa\circ \SO(df).
\]
For $(p,r)\in \bbS^d\times(1/2,2)$ and $(v,t)\in T_p\bbS^d\times\R$,
\[
\begin{split}
|A(v,t)|^2 &= |t\, \kappa\circ f  + r\, d\kappa\circ \SO(df)(v)|^2 \\
&= |t\, \kappa\circ f|^2  + r^2\, |d\kappa\circ \SO(df)(v)|^2 \\
&= t^2 + r^2 |v|^2 \\
&= |(v,t)|^2,
\end{split}
\]
where in the passage to the second line we used the fact that (for this specific embedding) the images of $\kappa$ and $d\kappa$ are orthogonal,  in the passage to the third line we used the fact that the image of $\kappa$ has norm one and that $d\kappa$ is an isometric immersion, whereas the passage to the fourth line is the definition of $\tilde\s$.
Thus, $A\in \O(\tilde\s,\euc)$, and it is orientation-preserving by construction.
It follows that
\[
\dist(d(\Phi\circ \hf),\SO(\tilde\s,\euc)) \le |r\,d\kappa\circ(df - \SO(df))| \le 2 \, \dist(df,\SO(f^*\s,\s)),
\]
and therefore
\beq
 \|\dist(d(\Phi\circ \hf),\SO(\tilde\s,\euc))\|_{L^p(\M;\R^{d+1})} \le
 c\,  \|\dist(df,\SO(f^*\s,\s))\|_{L^p(\M)},
\label{eq:FJMcodim2}
\eeq
where $c = c(p,d)$.

By the above discussion, there exists an isometry $\Psi$ of $\bbS^d$, such that
\[
\hPsi\circ \Phi(p,r) = r\,\kappa(\Psi(p)),
\]
and hence
\[
d(\hPsi\circ \Phi) = (\kappa\circ\Psi)\, dr + r\,d\kappa\circ d\Psi,
\]
i.e.,
\[
d(\Phi\circ \hf) - d(\hPsi\circ \Phi) = (\kappa\circ f -  \kappa\circ\Psi)\, dr + r\, (d\kappa\circ df -  d\kappa\circ d\Psi),
\]
where as noted before, $\kappa$ and $d\kappa$ cannot be factored out.
The two terms on the right-hand side are orthogonal in $(T\bbS^d\oplus\R)\otimes\R^{d+1}$
(with respect to the metrics $\S$ and $\euc$) since $df$ and $d\Psi$ act of the $T\bbS^d$ part and $dr$ acts on the $\R$ part.
It follows that 
\[
|d(\Phi\circ \hf) - d(\hPsi\circ \Phi)| \ge r\, |d\kappa\circ df - d\kappa\circ d\Psi| \ge 
\tfrac12 |d\kappa\circ df - d\kappa\circ d\Psi|,
\]
from which it follows that
\beq
\|d\kappa\circ df - d\kappa\circ d\Psi\|_{L^p(\M;\R^{d+1})} \le c'\, \|d(\Phi\circ \hf) - d(\hPsi\circ \Phi)\|_{L^p(\tilde{\M};\R^{d+1})} ,
\label{eq:FJMcodim3}
\eeq
where $c' = c'(p,d)$.
Combining \eqref{eq:FJMcodim1}, \eqref{eq:FJMcodim2} and \eqref{eq:FJMcodim3}, together with the Poincar\'e inequality (applicable since the image of $\kappa$ is bounded), we obtain the desired result.
\end{proof}
%%%%%%%%%%%%%

Recently, Conti, Dolzmann and M\"uller \cite{CDM24} proved the quantitative rigidity property for an arbitrary closed $(\S,\s)$ for the case $\M=\S$:
\begin{theorem}[Conti--Dolzmann--M\"uller]
\label{thm:CDM}
Let Let $(\S,\s)$ be a compact, oriented Riemannian manifold, and let $\kappa:\S\to\R^{D}$ be an isometric embedding. 
There exists for every $f\in W^{1,p}(\S;\S)$ an isometry $\Psi$ of $\S$, such that
\[
\|\kappa\circ f - \kappa\circ \Psi\|_{W^{1,p}(\S;\R^{D})} \lesssim \|\dist(df,\SO(\s,f^*\s))\|_{L^p(\S)},
\]  
for any $p\in (1,\infty)$.
\end{theorem}

%%%%%%%%%%%%%
\begin{corollary}
\label{corr:sphere}
Let $(\S,\s)$ be a round $d$-sphere, or let $\M=\S$. Let $\feps\in W^{1,2}(\M;\S)$ satisfy $\Eeps(\feps) \lesssim 1$. Then, $\feps$ has a subsequence (not relabeled) satisfying $\feps\to(\Psi,\nabs u)$ for some $\Psi\in\Iso(\S)$ and $u\in W^{1,2}\mathfrak{X}(\M)$. 
\end{corollary}
%%%%%%%%%%%%%

%%%%%%%%%%%%%%%%%%%%%%%%%%%%%%%%%%%%%%%%%%%%%
\section{Estimating the minimal energy}
\label{sec:energy}

The limit energy is the function $E_0:W^{1,2}\mathfrak{X}(\M)\to\R$, given by
\[
E_0(u) =  \tfrac{1}{4} \|\mathcal{L}_u\s  - \h\|_{L^2(\M)}^2.
\]
In this section, we analyze its minimum, and relate it in geometric terms to the asymptotic incompatibility between the metrics $\geps$ and $\s$. Adopting the notation in \cite{KL25}, we denote by $\mathscr{C}^{1,1}(\M)$ the set of symmetric sections of $T^*\M\otimes T^*\M$. The asymptotic metric incompatibility is captured by the  tensor field $\h\in L^2\mathscr{C}^{1,1}(\M)$; viewing the map $\e\mapsto\geps$ as a map $[0,\e_0)\to L^2\mathscr{C}^{1,1}(\M)$, we may express $\h$ as 
\[
\h = \left.\deriv{}{\e}\right|_{\e=0} \geps.
\]
The goal is to relate the minimum of $E_0$ (recall that a $\Gamma$-limit always has a minimizer) with geometric properties of $\h$.

Clearly, 
\[
\min E_0 = 0
\]
if and only if $\h$ belongs to the image of the deformation operator (also known as the Killing operator), $\Def: W^{1,2}\mathfrak{X}(\M)\to L^2\mathscr{C}^{1,1}(\M)$,
\[
\Def: u\mapsto \mathcal{L}_u\s
\]
The space $L^2\mathscr{C}^{1,1}(\M)$ of symmetric tensors has been studied extensively, starting with Berger and Ebin \cite{BE69}. 
The characterization of the image of the deformation operator has its own longstanding history in the context of the so-called Saint-Venant problem \cite{Gur73}, and was also considered by Calabi in a more general Riemannian context \cite{Cal61}.  It is only recently that a complete characterization was obtained using techniques from overdetermined pseudodifferential boundary-value problems \cite{KL25}. For conciseness, we will only mention those results that are directly applicable to the present work.

The image $\mathscr{R}(\Def)$ of the deformation operator is a closed subspace of $L^2\mathscr{C}^{1,1}(\M)$, implying an orthogonal decomposition,
\[
L^2\mathscr{C}^{1,1}(\M)  = \mathscr{R}(\Def) \oplus (\mathscr{R}(\Def))^\perp,
\] 
along with continuous projections, $P_\parallel: L^2\mathscr{C}^{1,1}(\M)\to \mathscr{R}(\Def)$ and $P_\perp: L^2\mathscr{C}^{1,1}(\M)\to (\mathscr{R}(\Def))^\perp$. Thus,
\[
\min E_0 = \tfrac14 \|P_\perp\h\|_{L^2(\M)}^2.
\] 
The fact that the minimal energy depends only on the component of the metric discrepancy orthogonal to the range of the deformation operator has a geometric interpretation: 
Our elemental assumption is that the family of body manifolds can be embedded in an open subset $\M$ of the space manifold $\S$, such that the Riemannian metrics $\geps$ of the body manifolds deviate from the metric $\s$ of the space manifold (restricted to $\M$) by an $O(\e)$ correction.
In other words, for any $\e$ we fixed the (inclusion) embedding
\[
\iota:(\M,\g_\e) \to (\S,\s).
\]
This is a choice of embedding that has the correct energy scaling, but it is not the only choice.
Indeed, there is a large gauge freedom: for any other map
\[
\Phi_\e :(\M,\g_\e) \to (\S,\s),
\]
whose distance from $\iota$ is $O(\e)$ (say in $C^1$), we still obtain that the metric discrepancy is $O(\e)$, that is
\[
\|\g_\e - \Phi_\e^* \s\|_{L^2(\M)} = O(\e).
\]
By fixing $\iota$ we thus made a choice of gauge that might not be optimal.
The projection $P_\parallel \h = \mathcal{L}_u\s$ for some $u\in W^{1,2}\mathfrak{X}(\M)$ contains exactly the information of the (asymptotically) optimal gauge: 
if we define $\Phi_\e$ to be the flow of the vector field $-u$ up to time $\e$, then we obtain that the discrepancy of the metrics is, to leading order, $P_\perp\h$, rather than $\h$.
Thus, by choosing a more flexible setting, we could have assumed from the outset that the identification of the body manifolds with a submanifold of space is optimal in the sense that $\h$ is orthogonal to the range of the deformation operator.

We now relate the infimal energy to a (linearized) curvature discrepancy.
Denote by $\Rm_\g \in \Gamma(\wedge^2T^*\M\otimes \wedge^2T^*\M)$ the $(4,0)$-Riemann curvature tensor corresponding to a metric $\g$ on $\M$, and by $\Re_\g\in\Gamma(\Hom(T\M\otimes T\M))$ the associated curvature tensor given by raising two indices, 
\[
\ip{\Re_\g(X,Y),(W,Z)}_\g = \Rm_\g(X,W,Y,Z). 
\]
In other words, if a coordinate representation of $\Rm_\g$ is $\mathcal{R}_{ijkl}$, then the coordinate representation of $\Re_\g$ is $\mathcal{R}_{i\,\,j}^{\,\, k\,\,l}$.
Viewing the map $\g\mapsto \Re_\g$ as a map $\mathscr{C}^{1,1}(\M)\to \Gamma(\Hom(T\M\otimes T\M))$, its derivative is the second-order linear differential operator 
\[
\dot{\Re}_\g\sigma = \left.\deriv{}{t}\right|_{t=0} \Re_{\g + t\,\sigma}.
\]
For $\g$ Euclidean, $\dot{\Re}_\g$ coincides with the so-called curl-curl operator; for arbitrary metrics, it corresponds to a Riemannian version of the curl-curl operator plus a  tensorial term.

The overdetermined elliptic systems studied in \cite{KL25} give rise to the following estimate for $\sigma\in W^{s,2}\mathscr{C}^{1,1}(\M)$, $s\ge 2$,
\[
\|\sigma\|_{W^{s,2}(\M)} \lesssim \|\delnabs\sigma\|_{W^{s-1,2}(\M)} + \|\Pn\sigma\|_{W^{s-1/2,2}(\partial\M)}
+ \|\dot{\Re}_\s\sigma\|_{W^{s-2,2}(\M)} + \|\mathscr{P}\sigma\|_{L^2(\M)},
\] 
where $\delnabs:\mathscr{C}^{1,1}(\M) \to \Omega^1(\M)$ is the covariant codifferential, $\Pn$ is the normal projection boundary operator, and $\mathscr{P}$ is the orthogonal projection on the kernel of $\delnabs\oplus\Pn\oplus\dot{\Re}_\s$, which is finite-dimensional and consists of smooth functions. 
This kernel is denoted by $\mathscr{B}^1_1(\M,\s)$ in \cite{KL25}. For $s\ge1$, the kernel of $\delnabs\oplus\Pn$ coincides with the subspace orthogonal to $\mathscr{R}(\Def)$. 
Our minimization problem concerns the case $s=0$, for which the differential operators can be interpreted in a distributional sense, but not the boundary operator, which is a trace operator. 
This can be overcome by restricting the inequality to $\sigma \in (\mathscr{R}(\Def))^\perp$ 
(since $(\mathscr{R}(\Def))^\perp \cap W^{1,2}\mathscr{C}^{1,1}(\M)\subset \ker\Pn$), where a more delicate analysis yields
\[
\|\sigma\|_{L^2(\M)} \lesssim  \|\dot{\Re}_\s\sigma\|_{W^{-2,2}(\M)} + \|\mathscr{P}\sigma\|_{L^2(\M)},
\]
where $\mathscr{P}$ is to be interpreted as a projection onto $\ker \dot{\Re}_\s \cap (\mathscr{R}(\Def))^\perp$
(see \cite[Prop.~4.1]{KL25}, substituting $A_k =\dot{\Re}_\s$ and $A_{k-1}\mathcal{P}_{k-1} = \mathscr{P}$). The reverse inequality
is an immediate consequence of $\dot{\Re}_\s$ being a bounded operator $L^2\to W^{-2,2}$, vanishing on the image of $\mathscr{P}$, i.e.,
\[
\min E_0 = \tfrac14  \|P_\perp\h\|^2_{L^2(\M)} \simeq  \|\dot{\Re}_\s P_\perp \h\|^2_{W^{-2,2}(\M)} + \|\mathscr{P} \h\|^2_{L^2(\M)}.
\]
We have thus expressed the minimal energy as a sum of two terms.
The first term is the square of a norm of the derivative
\[
 \left.\deriv{}{\e}\right|_{\e=0} \Re_{\s + \e P_\perp \h},
\]
as $\s + \e P_\perp \h$ is the ``correct'' metric in terms of energy minimization after the right gauge choice. 
As proved in \cite{KL22}, for the particular case of a manifold of constant sectional curvature, $\dot{\Re}_\s$ annihilates the range of the deformation operator, and hence in that case (and only in that case),
\[
\dot{\Re}_\s P_\perp \h = \dot{\Re}_\s \h = \left.\deriv{}{\e}\right|_{\e=0} \Re_{\geps} =  \left.\deriv{}{\e}\right|_{\e=0} \Re_{\s + \e \h}.
\]

The second term is a projection of the metric variation on the subspace $\mathscr{B}^1_1(\M,\s)$ of those variations that cannot be ``reduced" by a reparametrization, and that to leading order do not change the curvature. The results in \cite{KL25} establish that $\mathscr{B}^1_1(\M,\s)$ is finite-dimensional, but is inconclusive as to whether its dimension is a topological invariant or may also depend on the geometry.  It has long been known that for $(\M,\s)$ a simply-connected  Euclidean domain, $\mathscr{B}^1_1(\M,\s) = \{0\}$. In \cite[Thm.~4.5]{KL22} this result was extended to simply-connected manifolds having constant positive sectional curvature. A simple perturbative argument show that  $\mathscr{B}^1_1(\M,\s) = \{0\}$ for $\M$ simply-connected and $\s$ a small enough $C^2$-perturbation of a metric having constant non-negative curvature.
To conclude, we have established the following:
\begin{theorem}
\label{thm:energy_estimate}
Assume that $\M$ is a simply-connected, compact manifold (possibly with boundary) and $\s$ is a small enough $C^2$-perturbation of a metric having constant non-negative curvature (or, more generally, if $(\M,\s)$ satisfies $\mathscr{B}^1_1(\M,\s) = \{0\}$).
Then the minimal energy scales with the square of the $W^{-2,2}$-norm of the curvature variation of the gauge-corrected metric,
\[
\min E_0  \simeq  \left\|\left.\deriv{}{\e}\right|_{\e=0} \Re_{\s + \e P_\perp \h}\right\|^2_{W^{-2,2}(\M)}.
\]
In the case of constant sectional curvature, one can replace $P_\perp \h$ by $\h$.
\end{theorem}

This result, when the target manifold is Euclidean space, proves a linear Riemannian version of a conjecture 
by Lewicka and Mahadevan \cite{LM22}, whereby for a body manifold $(\M,\g)$ immersed in Euclidean space of same dimension, the infimal elastic energy scales like the square of the $W^{-2,2}$ norm of the curvature, as hinted already by the very rudimentary lower bound obtained in \cite{KS12}. 

%%%%%%%%%%%%%%%%%%%%%%%%%%%%%%%%%%%%%%%%%%%%%%
\section{Discussion: alternative displacement notion}
\label{sec:discussion}

As mentioned in the introduction, there are manifolds $(\S,\s)$ that do not satisfy the quantitative rigidity property of Definition~\ref{def:FJM}. 
Take for example a Riemannian manifold for which $\Iso(\S)=\{\id\}$, which has an open submanifold $\M'$ isomorphic to a Euclidean domain, and let $\M$ be a compactly-embedded open submanifold of $\M'$. Then there exist non-trivial isometric embeddings of $\M$ in $\M'$, thus violating the inequality in Definition~\ref{def:FJM}.
This example is not generic, as for a generic metric on a Riemannian manifold, any open subset would only have the inclusion map as an isometric immersion. Thus we expect the rigidity property to hold for these generic metrics, and also for symmetric spaces, in which ``all points are the same''.

However, if one restricts the definition and requires the rigidity to hold only in the case $\M = \S$, then this property is satisfied by any closed manifold, as was recently shown by Conti, Dolzmann and M\"uller \cite{CDM24}.  Thus, our limit theory holds for this case as well.

Moreover, we are  not familiar with any counterexamples for a closed manifold if one replaces in Definition~\ref{def:FJM} $\Psi\in \Iso(\S)$ by $\Psi:\M\to \S$ an isometric immersion. 
In a sense, the latter is the more natural assumption, as isometric immersions are the zero energy configurations of the energy $\|\dist(df,\SO(\s,f^*\s))\|_{L^2(\M)}$, and maps of asymptotically vanishing energy converge to isometric immersions \cite{KMS19}.

In order for such a weaker quantitative rigidity property to guarantee compactness property as in Proposition~\ref{prop:3.1}, we need to modify accordingly the definition of convergence of displacements (Definition~\ref{def:limit}), i.e., to require that $\Psi$ and $\Feps$ be isometric immersions of $\M$ into $\S$.
For $\S$ compact, the analysis in the rest of the paper carries through under this wider definition of convergence, however, some of the proofs become more technical due to the fact that the set of isometric immersions has a more complicated structure than $\Iso(\S)$. 
We now sketch the main adaptations that need to be done (the following are merely a proof outline): 
\begin{itemize} 
\item Unlike $\Iso(\S)$, the space of isometric immersions $\M\to \S$ is not necessarily a manifold. 
It is still finite-dimensional in the sense that, fixing a point $p_0\in \M$, an isometric immersion is uniquely determined by its value and its derivative at $p_0$. 
This suffices for the space of isometries to be compact.

\item The limit of configurations in this case is still a covariant derivative of a displacement vector field, $\nabs u$, as in Proposition~\ref{prop:J_is_a_gradient}.
The proof is similar, though more technical, as one cannot reduce it to the case of $\Feps = \iota$, since, even if all the $\Feps$ are embeddings, the composition $\Feps^{-1}\circ \hfeps$ is not defined.

\item Unlike Proposition~\ref{prop:limit_unique}, the limit is unique not modulo a Killing field of $(\S,\s)$, but modulo a Killing field of $(\M,\s)$, i.e., vector fields $u\in \mathfrak{X}(\M)$, satisfying $\mathcal{L}_u \s = 0$ (these are exactly the fields that do not affect the limit energy).
A key for proving this is proving a version of Lemma~\ref{lem:isometries} for isometric immersions $\Feps$ converging to some limit $\Psi$ at rate $\e$, which should hold because of the finite-dimensionality of the space of isometric immersions.

\item The $\Gamma$-convergence and compactness results remain unchanged.
\end{itemize}

\newcommand{\etalchar}[1]{$^{#1}$}
\providecommand{\bysame}{\leavevmode\hbox to3em{\hrulefill}\thinspace}
\providecommand{\MR}{\relax\ifhmode\unskip\space\fi MR }
% \MRhref is called by the amsart/book/proc definition of \MR.
\providecommand{\MRhref}[2]{%
  \href{http://www.ams.org/mathscinet-getitem?mr=#1}{#2}
}
\providecommand{\href}[2]{#2}

\end{document}